\documentclass[twoside,11pt]{article}
\usepackage[enableskew]{youngtab}
\usepackage{amsmath,amsfonts,amssymb,amsthm}
\usepackage{stmaryrd}
\usepackage{mathtools}
\usepackage{graphicx}
\usepackage[poly,arrow,curve,matrix]{xy}
\usepackage{wrapfig}
\usepackage{enumerate}

\newcommand{\id}{{\rm id}}

\newcommand{\Frac}{{\rm Frac}}
\newcommand{\lcd}{{\rm lcd}}
\newcommand{\norm}{{\sf norm}}
\newcommand{\join}{\vee}
\newcommand{\meet}{\wedge}
\renewcommand{\k}{\Bbbk}

\newcommand{\N}{{\mathbb N}}
\newcommand{\Z}{{\mathbb Z}}

\newcommand{\QQ}{{\mathbb Q}}
\newcommand{\RR}{{\mathbb R}}
\newcommand{\A}{{\mathcal A}}
\newcommand{\B}{{\mathcal B}}
\newcommand{\GG}{{\Gamma}}
\newcommand{\D}{{\mathcal D}}
\renewcommand{\O}{{\mathcal O}}
\newcommand{\F}{{\mathcal F}}
\newcommand{\I}{{\mathcal I}}

\newcommand{\R}{{\mathcal R}}

\renewcommand{\c}{{\mathbf c}}
\newcommand{\q}{{\mathbf q}}
\newcommand{\g}{{\mathfrak g}}
\newcommand{\m}{{\mathfrak m}}
\newcommand{\n}{{\mathfrak n}}
\renewcommand{\a}{{\mathfrak a}}

\newcommand{\Mod}{{\sf Mod}}
\newcommand{\GrMod}{{\sf Gr}}
\newcommand{\grmod}{{\sf gr}}
\newcommand{\Hom}{{\sf Hom}}
\newcommand{\HOM}{\underline{{\sf Hom}}}
\newcommand{\Ext}{{\sf Ext}}
\newcommand{\EXT}{\underline{\sf Ext}}
\newcommand{\opp}{{\rm opp}}
\newcommand{\tot}{{\rm tot}}
\newcommand{\supp}{{\sf supp}}
\newcommand{\pd}{{\sf projdim}}
\newcommand{\injdim}{\mathsf{injdim}}

\newcommand{\irr}{\underline{{\rm irr}}}
\newcommand{\FrMon}{{\sf fr}}
\newcommand{\StMon}{{\sf str}}
\newcommand{\inc}{\underline{{\rm inc}}}
\newcommand{\colim}{\displaystyle\lim_{\longrightarrow}}
\newcommand{\tq}{\,|\,}
\DeclareMathOperator{\Res}{Res}

\input xy
\xyoption{matrix}\xyoption{arrow}
\def\edge{\ar@{-}}
\def\dedge{\ar@{.}}

\newtheorem{suptheorem}{Theorem}

\newtheorem{supalinea}[suptheorem]{\S}

\newtheorem{subtheorem}{Theorem}[subsection]
\newtheorem{subproposition}[subtheorem]{Proposition}
\newtheorem{subdefinition}[subtheorem]{Definition}
\newtheorem{sublemma}[subtheorem]{Lemma}
\newtheorem{subexample}[subtheorem]{Example}

\newtheorem{subcorollary}[subtheorem]{Corollary}
\newtheorem{subremark}[subtheorem]{Remark}

\newcommand{\titre}{Twisted semigroup algebras.}

\newcommand{\fichier}{lrpz2-arg.tex}

\begin{document}

\title{{\vspace{-1.5cm} \bf \titre}}
\author{L. RIGAL, P. ZADUNAISKY\footnote{This work has been partially supported by the
projects ECOS-MINCyT action No. A10E04, MATHAmSud-NOCOSETA, UBACYT
20020100100475, PIP-CONICET 112-200801-00487 and PICT-2007-02182. The second author is a
CONICET fellow.}}
\date{}
\maketitle

\begin{abstract}
We study 2-cocycle twists, or equivalently Zhang twists, of semigroup algebras over a
field $\k$. If the underlying semigroup is affine, that is abelian, cancellative and
finitely generated, then $\mathsf{Spec}~\k[S]$ is an affine toric variety over $\k$, and
we refer to the twists of $\k[S]$ as \emph{quantum affine toric varieties}. We show that
every quantum affine toric varieties has a ``dense quantum torus'', in the sense that it
has a localization isomorphic to a quantum torus. We study quantum affine toric varieties
and show that many geometric regularity properties survive the deformation process.

\tableofcontents
\end{abstract}

\section*{Introduction.}

Fix a field $\k$. Let $S$ be a commutative semigroup with identity. Classically, one
associates to $S$ its semigroup algebra $\k[S]$ defined as follows: it is the
$\k$-vector space with basis $\{X^s \mid s\in S\}$ indexed by the elements of $S$,
equipped with the unique associative product such that $X^s \cdot X^{s'} = X^{s+s'}$ for
all $s,s' \in S$; it is $S$-graded in a natural way. In the present paper we are
interested in noncommutative deformations of $\k[S]$. Namely, we consider integral
algebras which respect the $S$-graded $\k$-vector space structure of $\k[S]$ but
with a possibly different associative product. It is not difficult to see that such a
deformation of $\k[S]$ is obtained twisting the original commutative multiplication by
means of a $2$-cocycle on $S$ with values in $\k^\times$. The deformation of $\k[S]$
corresponding to the $2$-cocycle $\alpha :S \times S \longrightarrow \k$ will be denoted
$\k^\alpha[S]$; its product, denoted $\cdot_\alpha$, satisfies $X^s \cdot_\alpha X^{s'} =
\alpha(s,s') X^{s+s'}$ for all $s,s' \in S$.

We will focus on the case where $S$ is an \emph{affine} semigroup, that is, $S$ is
finitely generated and isomorphic as semigroup to a subsemigroup of $\Z^n$ for some
positive integer $n$. In this case $\k[S]$ is a finitely generated $\k$-algebra and an
integral domain. Its maximal spectrum is an affine \emph{toric variety}, and if $\k$ is
algebraically closed then all affine toric varieties arise this way. With this fact in
mind we adopt the point of view of noncommutative algebraic geometry and consider the
$2$-cocycle twists of semigroup algebras as noncommutative analogues of affine toric
varieties, and refer to them as \emph{quantum affine toric varieties}. These objects 
were also studied in the widely circulated preprint [Ing], and some of our results are
similar to those found there; however, we work on a more general setting and our methods
are quite different.

Toric varieties have played a very important role in algebraic geometry in recent years.
From our point of view the more relevant result is the following: over the complex
numbers, any Schubert variety in a flag variety degenerates into a toric variety. This
result was proved by P. Caldero in [Cal]; the reader is referred to this article for a more
precise statement, related results and historical background. In an upcoming article we will
show that a similar result holds in the quantum context, namely that quantum Schubert
varieties degenerate to quantum affine toric varieties. In order to use this result to
study the properties of quantum Schubert varieties, we need to establish first the
properties of quantum affine toric varieties, which is the main objective of this article.
This program is carried out in full detail in the thesis [Zad].

Let $G$ be a simply connected, semisimple complex algebraic group and $\g$ its Lie
algebra. Lakshmibai and Reshetikin [LR] and Soibelmann [S] have defined quantum flag
varieties associated to $G$ as well as corresponding quantum Schubert subvarieties. This
is done again in the spirit of noncommutative geometry, that is to say, such a
\emph{quantum variety} is actually defined by means of a noncommutative algebra,
considered as its homogeneous coordinate ring. In the above setting, the quantum flag
varieties are defined as certain subalgebras of the Hopf dual of the quantum enveloping
algebra $U_q(\g)$, while the associated quantum Schubert varieties are quotients of the
latter obtained by means of certain quantum Demazure modules. For details on the general
construction, the reader may consult chapter 6 of [Zad].

As stated above, it can be shown that quantum Schubert varieties degenerate into
quantum toric varieties; by this we mean that the noncommutative $\k$-algebra which
defines a quantum Schubert variety may be equipped with a filtration whose
associated graded ring is isomorphic to a $\k$-algebra of the form $\k^\alpha[S]$. In [RZ]
we proved that this result holds for the more general class of quantum Richardson
varieties when the quantum flag variety is a quantum Grassmanians of type A.

The interest of the above result is that it allows to establish a number of properties for
quantum Schubert varieties by first proving them for quantum toric varieties and then
showing that the considered properties lift from the associated graded algebra to the
original one. \\

A guiding principle in noncommutative algebraic geometry is that if a geometric property
can be formulated in homological terms, then it should be stable by quantization, meaning
that if this property holds for a given coordinate ring it should also hold for
its quantum analogues. The properties that we have in mind are being Cohen-Macaulay,
Gorenstein or regular. Recall that these properties have been extended from the class of
commutative $\k$-algebras to the class of $\N$-graded not necessarily commutative algebras
by Artin, J\o rgensen, Van den Bergh, Yekutieli, Zhang and others. It turns out that the
guiding principle mentioned at the beginning of this paragraph can be given a very
concrete meaning for quantum affine toric varieties, and we finish this
introduction by discussing it in some detail since this is the main organizing principle
of the article.

Fix a subsemigroup $S \subseteq \N^{n+1}$ for some $n \geq 0$. Any $2$-cocycle deformation
$\k^\alpha[S]$ of its semigroup algebra has a natural $\Z^{n+1}$-grading, and since we wish
to study this noncommutative algebra in the context of the previous paragraph, we consider
also the $\N$-grading obtained by taking the total degree. This second grading is induced
by the group morphism $\phi: \Z^{n+1} \to \Z$ given by $\phi(a_0, \ldots, a_n) = a_0 +
\ldots + a_n$, in the sense of subsection \ref{ss-change-grading-group}. The regularity
properties of $\k^\alpha[S]$ are read from the category $\Z\GrMod \k^\alpha[S]$ of
$\Z$-graded $\k^\alpha[S]$-modules, and we expect them to be the same as those of $\k[S]$.
That this is indeed the case can be shown by mimicking the theory developed for the study
of the algebra  $\k[S]$, for example in [BH; chapter 6]. However, the relation between
both algebras can be made more precise: by a theorem of J. Zhang, the categories of
$\Z^{n+1}$-graded modules $\Z^{n+1}\GrMod~\k[S]$ and $\Z^{n+1}\GrMod~\k^\alpha[S]$ are
isomorphic; this
isomorphism is remarkably explicit and concrete, so information transfers between these
two categories in a straightforward way. We are then left to study the relation between
the categories of $\Z^{n+1}$ and $\Z$-graded $\k^\alpha[S]$-modules, and it turns out that
this is controlled by three functors, induced by the morphism $\phi$, with very good
homological properties. These functors fit into a diagram as follows
\[
\xymatrix@!C{
	\Z^{n+1}\GrMod\k[S] \ar@{<->}[r]^{\cong}  \ar@{->}[d]<-3ex>_{\phi_!}
	\ar@{<-}[d]|{\phi^*} \ar@{->}[d]<3ex>^{\phi_*}
	&\Z^{n+ 1}\GrMod \k^\alpha[S] \ar@{->}[d]<-3ex>_{\phi_!} \ar@{<-}[d]|{\phi^*}
	\ar@{->}[d]<3ex>^{\phi_*} \\
	\Z\GrMod \k[S]  
	& \Z\GrMod\k^\alpha[S]  \\
} 
\]
We use the vertical functors repeatedly to deduce information at the $\Z$-graded level
from the $\Z^{n+1}$-graded level, where the connection between the quantum and the
classical objects is very explicit. In this way we effectively deduce the regularity
properties of $\k^\alpha[S]$ from those of $\k[S]$.\\

The paper is organized as follows. 

Section \ref{section-prelim} establishes, in a fairly general setting, the fundamental
properties that we need on gradings and twistings. In the first subsection, generalities
are recalled, including local cohomology for noncommutative graded algebras. In the
second, the notion of a Zhang twist is introduced following [Z] and further properties
concerning the behavior of classical homological invariants with respect to such twists
are established. This notion, more general than twistings by $2$-cocycles, turns out to be
the proper one for our purposes. The last subsection deals with change of gradings over
the algebra. We consider a $G$-graded algebra $A$ and a group morphism $\phi : G
\longrightarrow H$. This morphism induces an $H$-grading over $A$, and an adjoint triple
$(\phi_!, \phi^*, \phi_*)$ relating the categories of $G$ and $H$-graded modules. These
functors are the main tools we use to transfer homological information between both
categories. \\

In section \ref{section-regularity-grading-twist}, we focus on the case where $G=\Z^{r+1},
H=\Z$ and $\phi : \Z^{r+1} \longrightarrow \Z$ sends an $r+1$-tuple to the sum of its
entries. In this context, the algebra $A$ is $\Z$-graded and the usual notions of
regularity from noncommutative algebraic geometry make sense. In the first subsection, we
study how the properties of being Cohen-Macaulay, Gorenstein or regular, which a priori
concern the category $\Z\GrMod A$, can actually be read in $\Z^n\GrMod A$. We also study
the behavior of the same regularity conditions with respect to twistings. The main result
of this section is Theorem \ref{stabilite-par-twist} which, roughly speaking, asserts that
a regularity property is true for $A$ if and only if it is true for any twist of $A$. In
the second subsection, similar questions are treated at the level of the derived
categories of modules.\\

In Section \ref{tsa} we study the class of algebras that were our original motivation:
quantum affine toric varieties. The first subsection collects basic
facts on twisting of semigroup algebras by $2$-cocycles. The second subsection restricts
the point of view to the case where the semigroup is affine. It is shown  that for such a
semigroup $S$ and for any $2$-cocycle $\alpha$, the algebra $\k^\alpha[S]$ is indeed a
Zhang twist of $\k[S]$, in particular Theorem \ref{stabilite-par-twist} applies. In the
same subsection we establish a decomposition statement which asserts that a quantum
affine toric variety whose underlying semigroup is normal is the intersection of a certain
family of subalgebras of its division ring of fractions, each isomorphic to a quantum
space localized at some of its generators, see Proposition \ref{decomposition}. As a
consequence we get a characterization by means of the underlying semigroup of those
quantum toric varieties which are normal domains, i.e. maximal orders in their division
ring of fractions, see Corollary \ref{corollaire-OM}. In the last subsection our
attention restricts to the case of \emph{twisted lattice algebras}. These are examples of
quantum affine toric varieties where the underlying semigroup is built from a certain
finite distributive lattice. These algebras arise naturally as degenerations of quantum
analogues of Richardson varieties in the quantum grassmannian and more generally of
symmetric quantum graded algebras with a straightening law, which were the object of study
of [RZ].

We would like to thank Andrea Solotar for her help and suggestions for the organization of
this article.

\paragraph{Conventions and notation.} Throughout, $\k$ denotes a field, and $G$ and $H$ are
commutative groups. We use additive notation.

\section{Preliminaries on gradings and twistings.} \label{section-prelim}

\subsection{Basic results.} \label{ss-br} 
Let $A$ be a $\k$-algebra. A $G$-grading on $A$ is a direct sum decomposition of $A$ as a
$\k$-vector space $A=\bigoplus_{g\in G} A_g$, such that $A_g A_{g'} \subseteq A_{g+g'}$
for all $g, g' \in G$.  We then say that $A$ is a $G$-graded algebra.

In this context, our attention will focus on the category of $G$-graded left $A$-modules,
which is the subject of study of [NV; Chap. A]. A $G$-graded left $A$-module is a left
$A$-module $M$ together with a direct sum decomposition $M = \bigoplus_{g\in G} M_g$ as
$\k$-vector space such that $A_g M_{g'} \subseteq M_{g+g'}$ for all $g, g' \in G$. The
spaces $M_g$ are called the homogeneous components of $M$; the support of $M$ is the set
$\supp M = \{g \mid M_g \neq 0\}$.

Fix $g\in G$. A morphism $f : M \longrightarrow N$ of $A$-modules is said to be
homogeneous of degree $g$ if $f(M_{g'}) \subseteq M_{g'+g}$ for all $g' \in G$. 
We denote by $G\GrMod A$ the category whose objects are $G$-graded $A$-modules and whose
morphisms are homogeneous $A$-module morphisms of degree $0$. It is easily verified that
$G\GrMod A$ is an abelian category with arbitrary products and coproducts, see [NV; \S
I.1]. The $g$-suspension functor $\Sigma_g: G\GrMod A \to G\GrMod A$ sends a graded module
$M$ to $M[g]$, defined as the graded module with the same underlying module structure as
$M$ and grading given by $M[g]_{g'} = M_{g'+g}$, while leaving morphisms unchanged. This
is an autoequivalence of the category of $G$-graded modules, in particular it is exact and
preserves projective and injective objects.

General homological methods apply to show that the category $G\GrMod A$ has enough
injective and projective objects. The projective and injective dimensions of a graded
module $M$ will be denoted by $G\pd M$ and $G\injdim M$ respectively.\\

Given objects $M, N$ of $G\GrMod A$ and $g\in G$, an $A$-module morphism $f: M
\longrightarrow N$ is homogeneous of degree $g$ if and only if it belongs to
$\Hom_{G\GrMod A}(M,N[g])$. With this in mind we set
\[
\HOM_{G\GrMod A}(M,N) = \bigoplus_{g\in G} \Hom_{G\GrMod A}(M,N[g]) \subseteq \Hom_A(M,N),
\]
which makes $\HOM_{G\GrMod A}(M,N)$ into a $G$-graded $\k$-vector space. This
inclusion is strict in general but it is an equality if $M$ is finitely generated, as the
following lemma states. For a proof see [NV; Corollary I.2.11].

\begin{sublemma} \label{HOM-Hom} -- 
Let $A$ be a $G$-graded $\k$-algebra and $M,N$ objects of $G\GrMod A$. If $M$ is a
finitely generated $A$-module, then $\HOM_{G\GrMod A}(M,N) = \Hom_A(M,N)$.
\end{sublemma}

For $i\in\N$, we denote the right derived functors of $\HOM_{G\GrMod A}(-,-)$ by
$\EXT^i_{G\GrMod A}(-,-)$. Lemma \ref{HOM-Hom} extends to the following result,
see [NV; Corollary I.2.12].  

\begin{sublemma} \label{EXT-Ext} -- 
Let $A$ be a left noetherian $G$-graded $\k$-algebra and $M,N$ objects of $G\GrMod A$. If
$M$ is a finitely generated $A$-module, then there exists a $\k$-vector space isomorphism
$\EXT^i_{G\GrMod A}(M,N) \cong \Ext^i_A(M,N)$ for all $i \in \N$.
\end{sublemma}

Our next aim is to introduce local cohomology in the present context. The definition is
completely analogous to that of local cohomology functors for commutative rings, and the
proof found in [BS, chapter 1] adapt to our context almost verbatim. We fix a
$G$-graded ideal $\a$ of $A$. The \emph{torsion functor} associated to $\a$, denoted by
\[
\Gamma_{\a,G} : G\GrMod A \longrightarrow G\GrMod A,
\]
is defined on objects as
\[
\Gamma_{\a,G}(M) = \{m\in M \tq \a^n m = 0 \mbox{ for } n \gg 0\} \subseteq M,
\]
and sends morphism $M \stackrel{f}{\longrightarrow} N$ to its restriction 
$\Gamma_{\a,G}(M) \stackrel{\Gamma_{\a,G}(f)}{\longrightarrow} \Gamma_{\a,G}(N)$. 
One can check that $\Gamma_{\a,G}$ is left exact. We denote its $i$-th right derived
functor by $H_{\a,G}^i$, and refer to it as the $i$-th local cohomology functor.

On the other hand, we consider the functor
\[
\colim \HOM_{G\GrMod A}(A/\a^n,-) : G\GrMod A  \longrightarrow  G\GrMod A.
\]
It is easy to check that it is left exact and naturally isomorphic to $\Gamma_{\a,G}$.
Standard homological algebra then yields for each $i \in \N$ a natural isomorphism
\[
H_{\a,G}^i \cong \colim \EXT_{G\GrMod A }^i(A/\a^n,-).
\]

\begin{subdefinition} -- \label{lc-def}
We define the local cohomological dimension of $A$ relative to its $G$-grading and to the
ideal $\a$, denoted $\lcd_{\a,G}(A)$, as the cohomological dimension of the functor
$\Gamma_{\a,G}$. That is $\lcd_{\a,G}(A)$ is the least integer $d$ such that
$H_{\a,G}^i(M)=0$ for all integers $i > d$ and all object $M$ of $G\GrMod A$ if such an
integer exists, or $+ \infty$ otherwise.
\end{subdefinition}

\begin{subremark} -- \label{lc-and-suspension} \rm 
Fix $i\in\N$ and $g\in G$. Since suspension is exact and preserves injectives, the
families of functors $(\Sigma_g \circ H^i_{\a,G})_{i \geq 0}$ and $(H^i_{\a,G} \circ
\Sigma_g)_{i \geq 0}$ are universal $\partial$-functors. It is clear that $\Sigma_g \circ
\GG_{\a, G} = \Gamma_{\a,G} \circ \Sigma_g$, and so the general theory of
$\partial$-functors states that there exist isomorphisms $\Sigma_g \circ H^i_{\a,G} \cong
H^i_{\a,G} \circ \Sigma_g$ for all $i \geq 0$.
\end{subremark}

\subsection{Zhang Twists.} \label{twistings}
Throughout this subsection $A$ denotes a $G$-graded algebra. The material is mostly taken
from [Z], where the reader may find the missing proofs. \\

\begin{subdefinition} ([Z; Definitions 2.1, 4.1]) --
\label{LTS}
A \emph{left twisting system} on $A$ over $G$ is a family of graded $\k$-linear
automorphisms $\tau = \{\tau_g| g \in G\}$ such that for any $g,g',g''\in G$ and any
$a,a'\in A$, homogeneous of degree $g$ and $g'$ respectively,
\[
\tau_{g''}(\tau_{g'}(a) a') = \tau_{g'+g''}(a) \tau_{g''}(a').
\]
A \emph{right twisting system} is similar, but the previous condition is replaced by
\[
\tau_{g''}(a \tau_g(a')) = \tau_{g''}(a)\tau_{g''+g}(a').
\]
A \emph{normalized} left, resp. right, twisting system on $A$ over $G$ is a left, resp.
right twisting system on $A$ over $G$ such that $\tau_0(1)=1$.
\end{subdefinition}
It is easy to see that if $\tau$ is a left twisting system for $A$ over $G$, then it is a
right twisting system for $A^\opp$ over $G$. This shows that every theorem on left
twistings has an analogue for right twistings, so we only state the left side versions.
If $\tau = \{\tau_g| g \in G\}$ is a normalized left twisting system on $A$ over $G$, then
$\tau_g(1)=1$ for all $g\in G$, and $\tau_0=\id$. See [Z; Proposition 2.2].

\begin{subtheorem} ([Z; Proposition/Definition 4.2]) --
\label{twisted-algebras}
Let $\tau$ be a left twisting system on $A$. The graded $\k$-vector space $A$ can be endowed
with an associative product denoted by $\circ$, given by
\[
a  \circ a' = \tau_{g'}(a) a'
\]
for all $g, g'\in G$, $a \in A_g$ and $a' \in A_{g'}$. With this product the $\k$-vector space
$A$ becomes a unital associative $G$-graded algebra whose unit is $\tau_0^{-1}(1)$. We
denote this algebra by $^\tau\! A$, and call it the \emph{left twisting} of $A$ by $\tau$.
\end{subtheorem}
We sometimes refer to ${}^\tau A$ as a Zhang twist of $A$. By [Z; Proposition 2.4], there
is no loss of generality if we only consider normalized twisting systems.

Two $G$-graded algebra structures over the underlying graded $\k$-vector space of $A$ are
twist-equivalent if one can be obtained from the other through a Zhang twist. This is
an equivalence relation by [Z; Proposition 2.5]. In particular, the following result
holds:
\begin{sublemma} 
\label{twisting-is-er} 
\label{twisted-eq-rel} --
For every left twisting system $\tau$ on $A$ there is a left twisting system $\tau'$ on
${}^\tau\! A$ such that ${}^{\tau'\!}({}^\tau\! A) = A$.
\end{sublemma}

Let $\tau$ be a left twisting system on $A$, and let $M$ be a $G$-graded $A$-module. Then
there exists a $G$-graded ${}^\tau \! A$-module whose underlying graded $\k$-vector space is
equal to that of $M$, with the action of an element $a \in {}^\tau\! A$ over $m \in M_g$,
with $g \in G$, is given by
\[
 a \circ m = \tau_g (a) \cdot m
\]
where $\cdot$ represents the action of $A$ on $M$. We denote this ${}^\tau\! A$-module by
${}^\tau M$. If $f: M \to N$ is a morphism of $G$-graded $A$-modules then the same
function defines a ${}^\tau\! A$-linear function ${}^\tau\! f: {}^\tau\!M \to {}^\tau\!N$.
This assignation defines a functor $\F_\tau: G\GrMod A \to G\GrMod {}^\tau\!A$.
The following result is crucial for us in the following sections.
\begin{subtheorem} ([Z; Theorem 3.1]) --
\label{category-isomorphism}
The functor $\F_\tau : G\GrMod A \longrightarrow G\GrMod {}^\tau\! A$ sending an object
$M$ to ${}^\tau\! M$ and leaving morphisms unchanged is an isomorphism of categories.
\end{subtheorem}
The isomorphism $\F_\tau$ is an isomorphism of \emph{abelian categories} and hence
preserves all homological properties of objects.

\begin{subtheorem} --
\label{left-right-TS}
Let $\tau$ be a left twisting system on $A$. For each $g\in G$, let $\nu_g$ be the
$\k$-linear map defined as
\[
	\nu_g(a) = \tau_{-(g+g')}\tau_{-g'}^{-1}(a).
\] 
for all $g' \in G$ and $a \in A_{g'}$. The following hold.
\begin{enumerate}[(i)]
	\item The set $\nu = \{\nu_g\}_{g\in G}$ is a right twisting system on $A$.
	\item The $\k$-linear map $\theta: {}^\tau\! A \to A^\nu$ sending $a \in {}^\tau\!
		A_g$ to $\theta(a) = \tau_{-g}(a)$ for each $g \in G$ is an isomorphism of
		$G$-graded algebras
	\item The change of rings functor $\Theta: G\GrMod ({}^\tau\! A)^\opp \to
		G\GrMod (A^\nu)^\opp$ induced by $\theta^{-1}$ is an isomorphism of
		categories, and $\theta: \Theta({}^\tau\! A) \to A^\nu$ is an isomorphism
		of right $A^\nu$-modules.
	\end{enumerate}
\end{subtheorem}
\begin{proof}
Points (i) and (ii) are [Z, Theorem 4.3]. Point (iii) follows at once. 
\end{proof}

\begin{subremark} -- \rm \label{left-right-cat-isomorphism}
Let $\tau$ be a left twisting system on $A$. Using Theorem \ref{category-isomorphism} and 
Theorem \ref{left-right-TS} we get that the category of right $A$-modules is isomorphic to
the category of right ${}^\tau\! A$-modules.
\end{subremark}

If $V$ is a graded subspace of $A$, then by abuse of notation we write ${}^\tau V$ when we
consider $V$ as a graded subspace of ${}^\tau A$. Notice that if $V$ is a left ideal of $A$
then it is a graded submodule, and so ${}^\tau V$ is a left ideal of ${}^\tau A$. The
following proposition, whose proof is straightforward, clarifies eventual ambiguities that
might arise due to this notation.
\begin{subproposition} -- \label{Z-twist-on-subalgebras}
Let $\tau$ be a left twisting system on $A$ over $G$ and denote by $\circ$ the product on
$A$ defined in Theorem \ref{twisted-algebras}. 
\begin{enumerate}
	\item Suppose $B$ is a graded subalgebra of $A$ such that $\tau_g(B) \subseteq B$
		for all $g \in G$. Then $({}^\tau\! B, \circ)$ is a subalgebra of
		$({}^\tau\! A, \circ)$. Furthermore $\tau$ induces by restriction a
		twisting system on $B$ over $G$, and the twist of $B$ by this induced
		system is equal to $({}^\tau\! B, \circ)$.
	
	\item Suppose that $\a$ is a graded two-sided ideal of $A$ such that $\tau_g(\a)
		\subseteq \a$ for all $g \in G$. Then ${}^\tau\! \a$ is a graded two-sided
		ideal of ${}^\tau\! A$.
	\end{enumerate}
\end{subproposition}

The following lemma shows that twisting commutes with local cohomology.

\begin{sublemma} -- \label{twisting-torsion}
Let $\tau$ be a left twisting system on $A$ over $G$ and let $\a$ be a graded
ideal of $A$ such that $\tau_g(\a) = \a$ for all $g \in G$. For all $i \in \N$ there are
natural isomorphisms $H_{{}^\tau\! \a,G}^i \circ \F_\tau\cong \F_\tau\circ H_{\a,G}^i$.
\end{sublemma}
 
\begin{proof} An easy verification shows that $\Gamma_{{}^\tau\! \a,G} \circ \F_\tau =
\F_\tau\circ \Gamma_{\a,G}$. Since $\F_\tau$ is exact and preserves injectives, the
families $(\F_\tau\circ H_{\a,G}^i)_{i\in\N}$ and
$(H_{^\tau\!\a,G}^i\circ\F_\tau)_{i\in\N}$ are universal $\partial$-functors.  
Hence the equality extends to give natural isomorphisms
$H_{{}^\tau\!\a,G}^i \circ \F_\tau\cong \F_\tau\circ H_{\a,G}^i$ for all $i \in\N$.\end{proof}

We will need the following result in future sections. It is adapted from [Z, section 5].
\begin{subproposition} -- \label{twisting-ext}
For all $i\in\N$, there are natural isomorphisms $\EXT^i_{G\GrMod {}^\tau\!
A}(-,{}^\tau\! A) \circ \F_\tau \cong \EXT^i_{G\GrMod A}(-,A)$ seen as functors from
$G\GrMod A$ to $G\GrMod \k$.
\end{subproposition}

\begin{proof} It is easy to see that $({}^\tau\! A)[g] \stackrel{\tau_{-g}}{\longrightarrow}
{}^\tau\! (A[g])$ is an  isomorphism in $G\GrMod({}^\tau\! A)$ for all $g \in G$, see [Z~;
Theorem 3.4]. It follows that we have natural isomorphisms $\HOM_{G\GrMod {}^\tau\!
A}(-,{}^\tau\! A) \circ \F_\tau \cong \HOM_{G\GrMod A}(-,A)$ and, $\F_\tau$ being an
exact functor, $(\EXT^i_{G\GrMod {}^\tau\! A}(-,{}^\tau\! A) \circ \F_\tau)_i$ is a
(contravariant) $\partial$-functor which is universal since $\F_\tau$ preserves
projectives. From this it follows that there exist two natural
isomorphisms $\EXT^i_{G\GrMod {}^\tau\! A}(-,{}^\tau\! A) \circ \F_\tau \cong
\EXT^i_{G\GrMod A}(-,A)$  for all $i\in\N$. \end{proof}

\subsection{Change of grading groups.} \label{ss-change-grading-group}

In this subsection, $G$ and $H$ are commutative groups and $\phi: G \longrightarrow H$ is
a group homomorphism; we view $\k$ as a $G$ and $H$-graded algebra concentrated in degree
$0$. 

The morphism $\phi$ induces three functors between the category of $G$-graded $\k$-vector
spaces and the category of $H$-graded $\k$-vector spaces, which we will now describe.
Let $M, M'$ be $G$-graded $\k$-vector spaces and $f: M \longrightarrow M'$ a
$G$-homogeneous morphism. Given $g \in G$ we denote by $f_g :M_g \longrightarrow M'_g$ 
the homogeneous component of $f$ of degree $g$, so that $f = \bigoplus_{g \in G}f_g$. The
same convention is adopted for morphisms of $H\GrMod A $.

\paragraph{- The \emph{shriek} functor, $\phi_! : G \GrMod \k \to H \GrMod \k$.} For every
$h \in H$, the homogeneous components of degree $h$ of $\phi_!(M)$ and $\phi_!(f)$ are
given by
\[
	\phi_!(M)_h = \bigoplus_{g \in \phi^{-1}(h)}M_g, \ \ \ \phi_!(f)_h = \bigoplus_{g
	\in \phi^{-1}(h)} f_g.
\]

\paragraph{- The \emph{lower star} functor, $\phi_*: G \GrMod \k \to H \GrMod \k$.} Once
again we give the homogeneous components of degree $h$ of $\phi_*(M)$ and $\phi_*(f)$:
\[
	\phi_*(M)_h = \prod_{g \in \phi^{-1}(h)}M_g, \ \ \ \phi_*(f)_h = \prod_{g
	\in \phi^{-1}(h)} f_g.
\]

\paragraph{- The \emph{upper star} functor, $\phi^* : H\GrMod \k \to G\GrMod \k$.}
Let $N, N'$ be objects of $H\GrMod A $, and $f: N \to N'$ a morphism of $H$-graded
modules. For every $g \in G$ the homogeneous components of $\phi^*(N)$ and $\phi^*(f)$ are
given by
\[
	\phi^*(N)_g = N_{\phi(g)}, \ \ \ \phi^*(f) = f_{\phi(g)}.
\]

Functoriality is easy to establish in all three cases. Notice that for every $h \in H$
there are several copies of the homogeneous component $N_h$ inside $\phi^*(N)$; in
particular if $g,g' \in \phi^{-1}(h)$ then $\phi^*(N)_g = \phi^*(N)_{g'} = N_h$. In order
to distinguish the elements in these two homogeneous components we will use the following
notational device: for every $n \in N$ we will denote by $n u_g$ the element $n \in
\phi^*(N)_g$. Notice that $n u_g$ makes sense only if $\deg n =\phi(g)$.

\begin{subremark} -- \label{faithful-exact} \rm
\begin{enumerate}
\item There is a natural transformation $\phi_! \to \phi_*$ induced by the natural inclusion
of the direct sum of a family of $\k$-vector spaces in its direct product. For a $G$-graded
$\k$-vector space $M$ the corresponding morphism  $\phi_!(M) \to \phi_*(M)$ is an
isomorphism if and only if $\supp (M) \cap \phi^{-1}(h)$ is a finite set for every $h \in
H$. Any $G$-graded $\k$-vector space with this property is called \emph{$\phi$-finite}. 
\item Let $L,M, N$ be graded $\k$-vector spaces and let
\[
0 \to L \to M \to N \to 0
\]
be a complex. This sequence is exact if and only if it is exact at each homogeneous
component. From this observation, it follows that the functors  $\phi^*$, $\phi_!$ and
$\phi_*$ are exact.
\item Let $M$ be an object of $G\GrMod\k$. It is clear by definition that if $\phi_!(M)=0$
or $\phi_*(M)=0$ then $M=0$. However this holds for $\phi^*$ only if $\phi$ is surjective.
\item From the above it follows that given any (co)chain complex $C$ in $G\GrMod\k$, the
complex $\phi_!(C)$ (resp. $\phi_*(C)$) is exact at homological degree $i\in\Z$ if and
only if $C$ is exact at degree $i$. This holds for $\phi^*$ only if $\phi$ is surjective.
\end{enumerate}
\end{subremark}

For the remainder of this subsection, we fix a $G$-graded $\k$-algebra $A$. Clearly, the
$\k$-vector space $\phi_!(A)$ is an $H$-graded $\k$-algebra. We will simply write $A$ in both
cases since the context will always make it clear which grading we are considering. Notice
that any $G$-homogeneous ideal of $A$ is also $H$-homogeneous.

Our aim now is to consider functors between the categories $G\GrMod A$ and $H\GrMod A$
naturally induced from $\phi_!$, $\phi_*$ and $\phi^*$. We will denote by $F_G : G\GrMod A
\to G\GrMod\k$ and $F_H : H\GrMod A \to H\GrMod\k$ the corresponding forgetful functors. 

Given $G$-graded $A$-modules $M, M'$ and a morphism $f: M \to M'$, we define $\phi_!(M)$
and $\phi_!(f)$ as before, and leave the action of $A$ on $M$ unchanged. We also define
$\phi_*(M)$ and $\phi_*(f)$ as before. The $A$-module structure on $\phi_*(M)$ is given as
follows: for every $h \in H$ and $(m_g)_{g \in \phi^{-1}(h)} \in \phi_*(M)_h$, the action
of a homogeneous element $a \in A_{g'}$ is given by $a (m_g)_{g \in \phi^{-1}(h)} =
(am_g)_{g \in \phi^{-1}(h)} \in \phi_*(M)_{\phi(g') + h}$. Finally for $H$-graded modules
$N, N'$ and a morphism $f: N \to N'$, we set $\phi^*(N)$ and $\phi^*(f)$ as before. The
action of a homogeneous element $a \in A_{g'}$ over $\phi^*(N)$ is defined as follows: for
each $n u_g \in \phi^*(N)_g$ we set $a (n u_g) = (an)u_{g'+g}$. The fact that the action
of $A$ is compatible with the gradings in each case is a routine verification.

Notice that by definition $\phi_! \circ F_G = F_H \circ \phi_!$ and similar equalities
hold for the other change of grading functors; it follows that Remark \ref{faithful-exact}
extends to the three functors defined at the level of graded $A$-modules. From this
point on $\phi_!, \phi_*$ and $\phi^*$ will denote the functors defined at the level
of graded $A$-modules; of course this includes the case $A = \k$. The main result
regarding them is the following, which asserts that $(\phi_!, \phi^*, \phi_*)$ is an
adjoint triple.

\begin{subproposition} -- \label{true-adjoints}
The functor $\phi^*$ is right adjoint to $\phi_!$ and left adjoint to $\phi_*$
\end{subproposition}

\begin{proof} Let $M$ be an object in $G \GrMod A$ and $N$ be an object in $H\GrMod A$. 
We first prove the existence of an isomorphism
\[
	\alpha: \Hom_{H\GrMod A}(\phi_!(M),N) \to \Hom_{G\GrMod A}(M,\phi^*(N)).
\]
Given $f: \phi_!(M) \to N$ of degree $0_H$, set $\alpha(f): M \to \phi^*(N)$ to be the
morphism given by the assignation $m \in M_g \mapsto f(m)u_g$. It is routine to check that
this is an $A$-linear morphism of degree $0_G$. Now given $f \in \Hom_{G\GrMod
A}(M,\phi^*(N))$, we define $\beta(f): \phi_!(M) \to N$ as follows: notice that it is
enough to define $\beta(f)$ over $G$-homogeneous elements of $M$ and that given
$m \in M_g$ we know that $f(m) = n u_g$ with $\deg n = \phi(g)$, so setting $\beta(f)(m) =
n$ we get an $H$-homogeneous $A$-linear morphism. Once again it is routine to check that
$\alpha$ and $\beta$ are inverses, and that they are in fact natural in both variables. 

Now we establish the existence of an isomorphism
\[
	\rho: \Hom_{H\GrMod A}(N,\phi_*(M)) \to \Hom_{G\GrMod A}(\phi^*(N), M).
\]
Fix $f: N \to \phi_*(M)$. For every $g \in G$ and every $nu_g \in \phi^*(N)_g$, we know
that $f(n) = (m_{g'})_{g' \in \phi^{-1}(\phi(g))}$, so we set $\rho(f)(n u_g) = m_g$. Its
inverse is given as follows: for every $f \in \Hom_{G\GrMod A}(\phi^*(N), M)$ and every $n
\in N_h$ with $h \in H$, we set $\varepsilon(f)(n) = (f(nu_g))_{g \in \phi^{-1}(h)}$ if $h
\in \operatorname{im} \phi$, and $0$ otherwise. We leave the task of checking the
good-definition and naturality of both morphisms, as well as the proof that $\rho$ and
$\varepsilon$ are inverses, to the interested reader.
\end{proof}

Since the change of grading functors are exact, Proposition \ref{true-adjoints} has the
following consequence. The proof is standard homological algebra, and can be found for
example in [W, Proposition 2.3.10].
\begin{subcorollary} -- \label{proj-inj-functors} The following properties hold:
\begin{enumerate}
\item the image by $\phi_!$ of a projective object is projective; 
\item the image by $\phi_*$ of an injectives object is injective;
\item the image by $\phi^*$ of an injective (resp. projective) object is injective (resp.
projective).
\end{enumerate}
\end{subcorollary}

The next two corollaries refine this last result.
\begin{subcorollary} -- \label{dim-and-!}
Let $M$ be an object of $G\GrMod A$.
\begin{enumerate}
\item $G\pd M = H\pd \phi_!(M)$; 
\item $G\injdim M  \le H\injdim \phi_!(M)$.
\end{enumerate}
\end{subcorollary}

\begin{proof} We point out that there is an isomorphism in $G\GrMod A$
\[
\phi^*(\phi_!(M)) \cong \bigoplus_{l\in\ker(\phi)} M[l], 
\]
from which it follows that $G\pd \phi^*(\phi_!(M)) = \sup \{G\pd M[l] \mid l \in \ker
\phi\}$; since suspension functors are auto-equivalences, they preserve projective
dimensions, so this supremum is equal to $G\pd M$. For the second inequality, suppose to
the contrary that $G\injdim M > H\injdim \phi_!M$. By Corollary \ref{proj-inj-functors} it
follows that $G\injdim \phi^*(\phi_!(M)) \le H\injdim \phi_!(M) < G\injdim M$, which can
not happen since $M$ is a direct summand of $\phi^*(\phi_!(M))$ in $G\GrMod A$.
\end{proof}

\begin{subcorollary} -- \label{dim-and-*} 
Let $M$ be an object of $G\GrMod A$. 
\begin{enumerate}
\item $G\injdim M  = H\injdim \phi_*(M)$;
\item $G\pd M \le H\pd\phi_*(M)$.
\end{enumerate}
\end{subcorollary}
\begin{proof}
For every $G$-graded $A$-module $M$ there is an isomorphism 
\[
\phi^*(\phi_* (M)) \cong \prod_{l\in\ker(\phi)} M[l].
\]
The proof follows the same pattern as that of Corollary \ref{dim-and-!}.
\end{proof}

Recalling from Remark \ref{faithful-exact} that $\phi_!(M) = \phi_*(M)$ if and only if
$M$ is $\phi$-finite, we can sum up our results in the following table:
\begin{table*}[h]
\begin{centering}
\begin{tabular}{|c|c|c|c|}
\hline
{} 
	& $\phi_!$
	& $\phi_*$
	& $\phi^*$ \\
\hline
$\pd$
	& is preserved
	& does not decrease 
	& does not increase \\
{}
	& {}
	& (preserved for $\phi$-finite modules)
	& {} \\\hline
$\injdim$
	& does not decrease
	& is preserved
	& does not increase \\
{}
	& (preserved for $\phi$-finite modules)
	& {} 
	& {}\\
\hline
\end{tabular} 
\end{centering}
\label{phi-finite}
\caption{Homological dimensions and functors associated to $\phi$.}
\end{table*}

\begin{subremark} -- \label{rmk-cebr} \rm
We point out that these inequalities are sharp. Consider $A = \k[x,x^{-1}]$ with the
obvious $\Z$-grading and let $\phi: \Z \to \{0\}$ be the trivial morphism. The category
$\Z \GrMod A$ is semisimple, and hence all its objects are projective and injective.
However $\phi_!(A)$ is not an injective $A$-module since it is not divisible as
$A$-module. Since $\phi^*(\phi_!(A))$ is injective, in this case $\phi_!$ increases
injective dimension and $\phi^*$ decreases it. Also $\phi_*(A)$ is not projective since
the element $(x^n)_{n \in \Z}$ is annihilated by $x-1 \in A$, but $\phi^*(\phi_*(A))$ is
projective, so $\phi_*$ increases projective dimension and $\phi^*$ decreases it in this
case.
\end{subremark}

We now study the relationship between the functors $\phi_!$, $\phi^*$ and both extension
and local cohomology functors. We denote by $G\grmod A$, resp. $H\grmod A$, the full
subcategory of $G\GrMod A $, resp. $H\GrMod A $, whose objects are finitely generated. 
Let us fix a $G$-homogeneous ideal $\a$ of $A$. The ideal $\a$ is also $H$-homogeneous, so
we have the two functors $\Gamma_{\a,G}$ and $\Gamma_{\a,H}$ as well as their respective
right derived functors. In order to study the relation between these derived functors we
need the following result.

\begin{subproposition} -- \label{extension-and-!-and-*} 
Suppose $A$ is left noetherian and let $M$ be an object of $G\GrMod A $. Then, for all
$i\in\N$, there are natural isomorphisms
\[
\phi_! \circ \EXT^i_{G\GrMod A }(-,M) \cong \EXT^i_{H\GrMod A }(-,\phi_!(M)) \circ \phi_!,
\]
as functors from $G\grmod(A)$ to $H\GrMod(\k)$.
\end{subproposition}

\begin{proof} Notice that, since $A$ is noetherian, the category $G\grmod(A)$ has enough
projectives. From lemma \ref{HOM-Hom} we see that
\[
\phi_! \circ \HOM_{G\GrMod A }(-,M) = \HOM_{H\GrMod A }(-,\phi_!(M)) \circ \phi_! .
\]
Since $\phi_!$ is exact and preserves projectives, the families of functors $(\phi_! \circ
\EXT^i_{G\GrMod A }(-,M))_{i \geq 0}$ and  $(\EXT^i_{H\GrMod A }(-,\phi_!(M)) \circ
\phi_!)_{i \geq 0}$ form universal contravariant homological $\partial$-functors, so the
equality extends to the desired isomorphisms by standard homological algebra. \end{proof}

We can now prove the following result.
\begin{subproposition} -- \label{cl-and-!-and-*} 
Fix $i\in\N$.
\begin{enumerate}
\item Suppose $A$ is left noetherian. There are natural isomorphisms of functors as
	follows:
\[
\phi_! \circ H^i_{\a,G} \cong H^i_{\a,H} \circ \phi_!. 
\]
\item There are natural isomorphisms of functors as follows:
\[
\phi^* \circ H^i_{\a,H} \cong H^i_{\a,G} \circ \phi^*. 
\]
\end{enumerate}
\end{subproposition}
\begin{proof}
Since $\phi_!$ is an exact functor the families $(\phi_! \circ H^i_{\a,G})_i$ and
$(H^i_{\a,H} \circ \phi_!)_i$ form cohomological $\partial$-functors. Clearly each functor
in the first family annihilates injective objects, so this is a universal
$\partial$-functor. If $I$ is an injective object in $G\GrMod A$ then
\[
H^i_{\a,H} \circ \phi_!(I) \cong  \varinjlim \EXT_{H\GrMod A}^i(A/\a^n,\phi_!(I)),
\]
and using Proposition \ref{extension-and-!-and-*} we conclude that $\EXT_{H\GrMod
A}^i(A/\a^n,\phi_!(I)) \cong \EXT_{G\GrMod A}^i(A/\a^n,I) = 0$. As a consequence,
$(H^i_{\a,H} \circ \phi_!)_i$ is also a universal $\partial$-functor. Since $\phi_! \circ
\Gamma_{\a,H} = \Gamma_{\a,G} \circ \phi_!$, standard homological algebra implies the
existence of the isomorphisms of the first item. The second item is proved
in the same way, using the fact that $\phi^*$ preserves injectives, see Corollary
\ref{proj-inj-functors}. \end{proof}

Recall that the local cohomological dimension relative to $G$ and $\a$, denote by
$\lcd_{\a,G} A$, is the cohomological dimension of the functor $\GG_{\a, G}$.
\begin{subcorollary} -- 
\label{lcd-group-change}
If $A$ is left noetherian, then $\lcd_{\a,G} A =\lcd_{\a,H} A$.
\end{subcorollary}

\begin{proof}
Let $i$ be an integer such that $H_{\a,H} ^i \neq 0$, i.e. there is an object
$M$ of $H\GrMod A $ such that $H_{\a,H}^i(M) \neq 0$. Since local cohomology commutes
with suspension functors, see Remark \ref{lc-and-suspension}, we may assume that the
component of degree $0_H$ of $H_{\a,H}^i(M)$ is non-zero. By the second item of
Proposition \ref{cl-and-!-and-*}, $H_{\a,G}^i(\phi^*(M))_{0_G} = H_{\a,H}^i(M)_{0_H} \neq
0$, so $\lcd_{\a,G} A \geq \lcd_{\a,H} A$. Since $A$ is noetherian we use the first item
of Proposition \ref{cl-and-!-and-*} and the fact that $\phi_!$ sends nonzero modules to
nonzero modules in an analogous fashion to get $\lcd_{\a,G} A  \leq \lcd_{\a,H} A$.  
\end{proof}

\section{Regularity of $\Z^{r+1}$-graded algebras and their Zhang
twists.}\label{section-regularity-grading-twist}

Throughout this section $A$ denotes a noetherian connected $\Z^{r+1}$-graded $\k$-algebra
for some $r\in\N$. In this context connected means that $\supp(A) \subseteq \N^{r+1}$ and
$A_{0} = \k$. We denote by $\m$ the ideal generated by all homogeneous elements of nonzero
degree, which is clearly the unique maximal graded ideal of $A$. Since $\m$ is a graded
ideal, we can consider the local cohomology functors $H_{\m, \Z^{r+1}}^i$ as defined in
section \ref{section-prelim}.

Let $\phi: \Z^{r+1} \to \Z$ be the group morphism $(a_0, \ldots, a_r) \mapsto a_0 + \ldots
+ a_r$ for all $(a_0, \ldots, a_r) \in\Z^{r+1}$. We are in position to apply the procedure
described in subsection \ref{ss-change-grading-group} and view $A$ as a $\Z$-graded
$\k$-algebra with the grading induced by $\phi$; notice that $A$ is also connected
with this grading. We can now consider the categories of $\Z^{r+1}$-graded and $\Z$-graded
$A$-modules.

Various noncommutative analogues of regularity conditions for algebraic varieties have been
introduced in the study of connected $\N$-graded algebras, for example the
AS-Cohen-Macaulay, AS-Gorenstein, AS-regular properties, or having dualizing complexes, to
name a few. We are interested in the following question: given a twisting system $\tau$ on
$A$ over $\Z^{r+1}$, is it true that $A$ and ${}^\tau A$ have the same regularity
properties when seen as $\N$-graded algebras? In order to answer this question we prove
that these properties, which a priori are read from the category of $\Z$-graded modules,
can also be read from the category of $\Z^{r+1}$-graded modules.

The section is organized as follows. In the first subsection we consider $\Z^{r+1}$-graded
analogues of the AS-Cohen-Macaulay, AS-Gorenstein and AS-regular properties, and show that
they are stable under change of grading and twisting. In the second subsection we recall
the definition of a balanced dualizing complex, and show that the property of having a
balanced dualizing complex is also invariant under twisting.

\subsection{Homological regularity conditions.} \label{hrc}
We start by introducing $\Z^{r+1}$-graded analogues of the AS-Cohen-Macaulay,
AS-Gorenstein and AS-regular conditions for connected $\N$-graded algebras. These
conditions can be found for example in the introduction to [JZ].
\begin{subdefinition} -- \label{def-reg-cond}
We keep the notation from the introduction to this section.
\begin{enumerate}
\item $A$ is called \emph{$\Z^{r+1}$-AS-Cohen-Macaulay} if there exists $n \in \N$
such that $H^i_{\m, \Z^{r+1}}(A) = 0$ and $H_{\m^\opp, \Z^{r+1}}^i(A) = 0$ for all $i \neq
n$.
\item $A$ is called \emph{left $\Z^{r+1}$-AS-Gorenstein} if it has finite left graded
injective dimension $n$ and there exists $\ell \in \Z^{r+1}$, called the \emph{Gorenstein
shift} of $A$, such that
\[
\EXT_{\Z^{r+1}\GrMod A}^i(\k,A) \cong 
\begin{cases} \k[\ell] & \mbox{ for } i = n \\ 0 & \mbox{ for } i \neq n\end{cases}
\]
as $\Z^{r+1}$-graded $A^\opp$-modules. We say $A$ is \emph{right} $\Z^{r+1}$-AS-Gorenstein
if $A^\opp$ is left $\Z^{r+1}$-AS-Gorenstein. Finally $A$ is
\emph{$\Z^{r+1}$-AS-Gorenstein} if both $A$ and $A^\opp$ are left
$\Z^{r+1}$-AS-Gorenstein,
with the same injective dimensions and Gorenstein shifts.

\item $A$ is called \emph{$\Z^{r+1}$-AS-regular} if it is
$\Z^{r+1}$-AS-Gorenstein, it has finite graded global dimension both on the left and on
the right, and both dimensions coincide.
\end{enumerate}
\end{subdefinition}
The usual notions of AS-Cohen-Macaulay, AS-Gorenstein and AS-regular $\N$-graded algebras
correspond to the case $r = 0$.
We will soon see that one may omit the $\Z^{r+1}$ of the definitions, since they are
stable by re-grading through a group morphism. For a precise statement see Theorem
\ref{AS-regrading}.

\begin{subremark} -- \rm
Let $B$ be a commutative noetherian connected $\N$-graded $\k$-algebra, with maximal
homogeneous ideal $\n$. The algebra $B$ is Cohen-Macaulay, resp. Gorenstein, in the
classical sense if and only if $B$ is AS-Cohen-Macaulay, resp. AS-Gorenstein, see [LR1;
Remark 2.1.10]. The same result holds for the AS-regular property, as we now show.

First suppose $B$ is regular in the classical sense [Mat, p. 157]. Then by [BH; ex.
2.2.24] it is regular if and only if $B_\m$ is local regular, and this implies that $B$
is a polynomial ring [Mat; ex. 19.1], which is easily seen to be AS-regular. On the other
hand if $B$ is AS-regular then by [Lev; 3.3] it has finite ungraded global dimension, so
it is regular by the Auslander-Buchsbaum-Serre theorem.
\end{subremark}

\begin{subremark} -- \rm
\label{one-dim}
If $M$ is an object of $\Z^{r+1}\GrMod A$ with $\dim_\k M = 1$ then it is clear that it is
isomorphic to $\k[\ell]$ for some $\ell \in \Z^{r+1}$ as an object of $\Z^{r+1}\GrMod A$.
Thus to prove that an algebra is left or right AS-Gorenstein it is enough to check that it
has finite graded injective dimension $n$ and that $\dim_\k \EXT_{\Z^{r+1}\GrMod A}^i(\k,
A) = \delta_{i,n}$.
\end{subremark}

The following two lemmas will be used in the proof of the invariance of the regularity
properties by change of grading.
\begin{sublemma} 
\label{spe-seq-and-csq}
If $A$ is both left and right $\Z^{r+1}$-AS-Gorenstein, it is $\Z^{r+1}$-AS-Gorenstein. 
\end{sublemma}
\begin{proof}
Let $\phi: \Z^{r+1} \to \Z$ be the morphism that sends an $r+1$-uple to the sum of its
coordinates. By Corollary \ref{dim-and-*} the $\Z^{r+1}$-graded injective dimension of $A$
is equal to that of $\phi_!(A)$ and by [Lev; Lemma 3.3] this is equal to the injective
dimension of $A$ over itself; thus the graded and ungraded injective dimensions of $A$
coincide. By [Zak; Lemma A], it is also equal to the injective dimension of $A^\opp$, so
the graded injective dimensions of $A$ and $A^\opp$ coincide. All that is left to do is to
prove that the left and right Gorenstein shifts of $A$ are equal.

Let $\ell$ and $r$ be the left and right Gorenstein shifts of $A$ respectively, and let $d$
be its left or right injective dimension. Using the $\Z^{r+1}$-graded analogue of the
Ischebeck spectral sequence first introduced in [Isch, Theorem 1.8], we obtain a spectral
sequence in $\Z^{r+1}\GrMod A$:
\[
E_2^{p,q}: \EXT_{A^\opp}^p(\EXT_A^{-q}(\k, A), A) \Rightarrow \mathbb H^{p+q}(\k) =
\begin{cases} \k \mbox{ if } p + q = 0 \\ 0 \mbox{ otherwise.}\end{cases}
\]
By hypothesis, page two of this spectral sequence has $\k[r - \ell]$ in position $(d,-d)$
and zero elsewhere, so $r$ and $\ell$ must be equal.
\end{proof}

\begin{sublemma} -- \label{gl-dim}
The global dimension of $\Z^{r+1}\GrMod A$ is equal to the projective dimension of the
trivial module $\k$ seen as an object of this category.
\end{sublemma}

\begin{proof} 
By Corollary \ref{dim-and-!}, the ungraded global dimension of $A$ is an upper bound for
the graded global dimension of $A$. The same result implies $\pd \k = \Z^{r+1}\pd \k = n$,
which is obviously a lower bound. By [Ber; Th\'eor\`eme 3.3] these three numbers are
equal.
\end{proof}

We now prove the result announced at the beginning of this subsection. Recall that we
denote by $\phi: \Z^{r+1}\to \Z$ the morphism that sends an $r+1$-uple to the sum of its
coordinates. 
\begin{subtheorem} -- \label{AS-regrading}
Let $A$ be a connected $\Z^{r+1}$-graded algebra.
\begin{enumerate}
\item $A$ is $\Z^{r+1}$-AS-Cohen-Macaulay if and only if $\phi_!(A)$ is
	$\Z$-AS-Cohen-Macaulay.

\item $A$ is $\Z^{r+1}$-AS-Gorenstein if and only if $\phi_!(A)$ is $\Z$-AS-Gorenstein.

\item $A$ is $\Z^{r+1}$-AS-regular if and only if $\phi_!(A)$ is $\Z$-AS-regular.
\end{enumerate}
\end{subtheorem}

\begin{proof} The first item follows immediately from the first statement in Proposition
\ref{cl-and-!-and-*}.

Since $\phi_!(A)$ is connected and noetherian, each of its homogeneous components is
finite dimensional, which implies that $A$ is $\phi$-finite, so $\phi_!(A) = \phi_*(A)$.
Using Corollary \ref{dim-and-*} we get $\Z^{r+1} \injdim A =
\Z \injdim \phi_!(A)$ and $\Z^{r+1} \injdim A^\opp = \Z \injdim
\phi_!(A)^\opp$. By Proposition \ref{extension-and-!-and-*} 
\[
\dim_\k \EXT_{\Z^{r+1}\GrMod A}^i(\k, A) = \dim_\k
\EXT_{\Z\GrMod \phi_!(A)}^i(\k, \phi_!(A))
\] 
for all $i \geq 0$, so using Remark \ref{one-dim} we see that $A$ is left
$\Z^{r+1}$-AS-Gorenstein if and only if $\phi_!(A)$ is left $\Z$-AS-Gorenstein; an
analogous argument works for $A^\opp$. The second item follows from Lemma
\ref{spe-seq-and-csq}.

By Corollary \ref{dim-and-!} the graded projective dimension of $\k$ as a graded $A$ or
$\phi_!(A)$-module coincide, so item 3 follows from item 2 and Lemma \ref{gl-dim}.
\end{proof}
\begin{subremark} -- \rm
Replacing $\phi$ with the trivial morphism $\Z^{r+1} \to \{0\}$ in the proof of Theorem
\ref{AS-regrading}, one can give necessary and sufficient conditions for $A$ to be
$\Z^{r+1}$-AS-Cohen-Macaulay, $\Z^{r+1}$-AS-Gorenstein or $\Z^{r+1}$-AS-regular in terms
of the category $\Mod~A$. In particular we see that if $A$ can be endowed with a different
connected grading such that $\m$ is its unique maximal graded ideal, then $A$ has any of
the aforementioned properties with respect to the new grading if and only if it has this
property with respect to the original grading. Hence the AS-properties may be seen as
invariants of the algebra $A$ and the ideal $\m$, regardless of the grading. This is the
point of view adopted in the article [BZ], where analogues of the AS-Gorenstein and
AS-regular properties are given for augmented algebras.
\end{subremark}

In view of this remark, from now on we omit the prefix ``$\Z^{r+1}$-'' from the
regularity properties mentioned there. We now focus on showing these properties are
invariant under twisting. We begin with an auxiliary lemma.

\begin{sublemma} -- \label{iso-pre-chi} Let $\tau$ be a normalized left twisting system on
	$A$ over $\Z^{r+1}$. For all $i\in\N$ there is a natural isomorphism of functors
	$\EXT^i_{\Z^{r+1}\GrMod A }(\k,-) \cong \EXT^i_{\Z^{r+1}\GrMod {}^\tau\!A}(\k,-)
	\circ \F_\tau$, seen as functors from $\Z^{r+1}\GrMod A$ to $\Z^{r+1}\GrMod\k$.
\end{sublemma}

\begin{proof} 
Let $M$ be an object of $\Z^{r+1}\GrMod A$. It is routine to check that, for every
$\xi \in \Z^{r+1}$, every $A$-linear function $f: \k \to M[\xi]$ also defines a ${}^\tau
A$-linear function from ${}^\tau \k$ to $({}^\tau M) [\xi]$; notice that twists and shifts
do not commute, so this is not true if we change $\k$ for any graded $A$-module. This
shows that $\HOM_{\Z^{r+1}\GrMod A}(\k,M)$ and $\HOM_{\Z^{r+1}\GrMod{}^\tau\!
A}(\k,{}^\tau\! M)$ coincide as $\k$-subspaces of $\HOM_{\Z^{r+1}\GrMod\k}(\k,M)$, and so
$\HOM_{\Z^{r+1}\GrMod A}(\k,-) = \HOM_{\Z^{r+1}\GrMod{}^\tau\!A}(\k,-) \circ \F_\tau$ as
functors from $\Z^{r+1}\GrMod A$ to $\Z^{r+1}\GrMod\k$. Since $\F_\tau$ preserves
injectives, the $\partial$-functor $(\EXT^i_{\Z^{r+1}\GrMod{}^\tau\!A}(\k,-) \circ
\F_\tau)_{i \geq 0}$ is universal and so by standard homological algebra we get a natural
isomorphism $\EXT^i_{\Z^{r+1}\GrMod A}(\k,-) \cong \EXT^i_{\Z^{r+1}\GrMod{}^\tau\!A}(\k,-)
\circ \F_\tau$ for all $i \geq 0$. 
\end{proof}

We are now ready to prove the result announced in the introduction to this subsection.
\begin{subtheorem} \label{stabilite-par-twist}
Let $\tau$ be a normalized left twisting system on $A$ over $\Z^{r+1}$. 
\begin{enumerate}
\item $A$ is AS-Cohen-Macaulay if and only if $\,{}^\tau\! A$ is AS-Cohen Macaulay.
\item $A$ is AS-Gorenstein if and only if $\,{}^\tau\! A$ is AS-Gorenstein.
\item $A$ is AS-regular if and only if $\,{}^\tau\! A$ is AS-regular.
\end{enumerate}
\end{subtheorem}

\begin{proof} 
Recall from Theorem \ref{left-right-TS} that there exists a right-twisting system $\nu$
on $A$ and an isomorphism $\theta: {}^\tau A \to A^\nu$. If we denote by $\Theta$ the
change of rings functor induced by $\theta^{-1}$, then $\Theta({}^\tau A) = A^\nu$.
\begin{enumerate}
\item Fix $i\in\N$. Lemma \ref{twisting-torsion} shows that $H_{\m,\Z^{r+1}}^i(A)=0$ if
and only if $H_{{}^\tau \m,\Z^{r+1}}^i({}^\tau\!A)=0$. Since $\Theta$ is an isomorphism of
categories, the proof of Lemma \ref{twisting-torsion} can be adapted to show that there is
a natural isomorphism \[H_{({}^\tau\!\m)^\opp,\Z^{r+1}}^i \cong \Theta^{-1} \circ
H_{(\m^{\nu})^\opp,\Z^{r+1}}^i \circ \Theta.\] Using this and the right-sided version of
Lemma \ref{twisting-torsion}, we get that $H_{\m^\opp,\Z^{r+1}}^i(A)=0$ if and only if
$H_{({}^\tau\m)^\opp,\Z^{r+1}}^i({}^\tau\!A)=0$. 

\item As we pointed out before, $\Z^{r+1}\GrMod A$ and $\Z^{r+1}\GrMod \,{}^\tau\!A$
are isomorphic categories, and the same holds for $\Z^{r+1}\GrMod A^\opp$ and
$\Z^{r+1}\GrMod ({}^\tau\!A)^\opp$. It follows that the left and right graded injective
dimensions of $A$ are equal to the corresponding graded injective dimensions of ${}^\tau
A$. By Lemma \ref{iso-pre-chi} there are graded $\k$-vector space isomorphisms
$\EXT^i_{\Z^{r+1}\GrMod{}^\tau\! A}(\k,{}^\tau\! A) \cong \EXT^i_{\Z^{r+1}\GrMod A}(\k,A)$
for all $i \geq 0$; arguing as for item 1, we also get isomorphisms of graded $\k$-vector
spaces $\EXT^i_{\Z^{r+1}\GrMod ({}^\tau\! A)^\opp}(\k,{}^\tau\! A) \cong
\EXT^i_{\Z^{r+1}\GrMod A^\opp}(\k,A)$. From this we deduce that $A$ is left, resp. right
AS-Gorenstein if and only if ${}^\tau A$ is left, resp. right, AS-Gorenstein; using Remark
\ref{one-dim}, the result follows by Lemma \ref{spe-seq-and-csq}.

\item The details concerning the AS-regular property are similar. \qedhere
\end{enumerate}
\end{proof}

\subsection{Dualizing complexes} \label{dc}
Recall that we denote by $A$ a noetherian connected $\N^{r+1}$-graded $\k$-algebra, and by
$\phi: \Z^{r+1} \to \Z$ the group morphism defined by the assignation $(a_0, \ldots, a_r)
\mapsto a_0 + \ldots + a_r$. We will consider $A$ as a connected $\N$-graded $\k$-algebra
with the grading induced by $\phi$. Our main concern in this subsection is
whether the existence of a balanced dualizing complex is a twisting invariant property.

We quickly recall some basic notions concerning dualizing complexes for the convenience of
the reader. Fix a noetherian connected $\N$-graded $\k$-algebra $B$ and set $B^e
= B\otimes B^\opp$. We denote by $\D(\Z\GrMod B)$ the derived category of $\Z\GrMod B$;
we denote by $\D^+(\Z\GrMod B),$ $\D^-(\Z\GrMod B)$ and $\D^b(\Z\GrMod B)$ the full
subcategories of $\D(\Z\GrMod B)$ whose objects are bounded below, bounded above
and bounded complexes, respectively. Of course we may replace $B$ with $B^\opp$ and $B^e$,
and we denote by $\Res_B: \D(\Z \GrMod B^e) \to \D(\Z \GrMod B)$ the obvious restriction
functor, with an analogous one $\Res_{B^\opp}$ for $B^\opp$. These functors preserve
projectives and injectives, see [Y; \S 2] for details. We write $\HOM_B$ instead of
$\HOM_{\Z\GrMod B}$, and denote its $i$-th derived functor by $\EXT_B^i$ for all $i \geq
0$.

The functor $\HOM_B: \Z\GrMod B^e \times \Z\GrMod B^e \longrightarrow \Z\GrMod B^e$
has a derived functor, $\R\HOM_B : \D(\Z\GrMod B^e) \times \D^+(\Z\GrMod B^e )
\longrightarrow \D(\Z\GrMod B^e)$ which can be calculated using resolutions by bimodules
that are projective as left $B$-modules on the first variable, or resolutions by bimodules
that are injective as left $B$-modules in the second variable. For more details, see [Y;
Theorem 2.2].

Denote by $\n$ the maximal graded ideal of $B$. The torsion functor $\Gamma_\n
: \Z\GrMod B  \longrightarrow \Z\GrMod B $ has a right derived functor $\R\Gamma_\n :
\D^+(\Z\GrMod B ) \longrightarrow \D^+(\Z\GrMod B )$. Notice that $\Gamma_\n$ sends
bimodules to bimodules, and so it induces an endofunctor of $\Z\GrMod B^e$. By abuse of
notation we denote the torsion functor on bimodules by $\Gamma_\n$ and its right derived
functor by $\R \Gamma_\n$. This abuse is justified since $\Gamma_\n$ commutes with
$\Res_B$ and this last functor preserves injectives, so for any left bounded complex of
bimodules $M^\bullet$ we have $\R\Gamma_\n(\Res_B(M^\bullet)) = \Res_B(\R
\Gamma_\n(M^\bullet))$.

Given a complex of $\Z$-graded $\k$-vector spaces $V^\bullet$ we denote by $(V^\bullet)'$ its
Matlis dual. Details on Matlis duals can be found in [VdB; \S 3]. The following definition
is found in [Y; Definitions 3.3 and 4.1] and [VdB; Definition 6.2].
\begin{subdefinition} -- \label{def-DC}
An object $R^\bullet$ of $\D^+(\Z\GrMod B^e)$ is called a \emph{dualizing complex} if it
satisfies the following conditions:
\begin{enumerate}
\item The objects $\Res_B(R^\bullet)$ of $\D(\Z\GrMod B )$ and $\Res_{B^\opp}(R^\bullet)$
of $\D(\Z\GrMod B^\opp)$ have finite injective dimension.

\item The objects $\Res_B(R^\bullet)$ of $\D(\Z\GrMod B )$ and $\Res_{B^\opp}(R^\bullet)$
of $\D(\Z\GrMod B^\opp)$ have finitely generated cohomology.

\item The natural morphisms $B^\opp \longrightarrow \R\HOM_B(R^\bullet,R^\bullet)$ and $B
\longrightarrow \R\HOM_{B^\opp} (R^\bullet,R^\bullet)$ are
isomorphisms in $\D^+(\Z\GrMod B^e)$.
\end{enumerate}
A dualizing complex $R^\bullet$ is said to be \emph{balanced} if $\R \Gamma_\n(R^\bullet)
\cong B' \cong \R \Gamma_{\n^\opp}(R^\bullet)$ in $\D(\Z\GrMod B^e)$. 
\end{subdefinition}
As shown in [Y; Corollary 4.21], if a balanced dualizing complex exists then it is
unique up to isomorphism, and given by $\R\Gamma_\n(B)'$.

\begin{subremark} -- \label{flcd-et-borne} \rm 
Recall that the local cohomological dimension of $B$, denoted by $\lcd_{\n, \Z}B$, is the
cohomological dimension of the functor $\Gamma_\n$ over $\Z\GrMod B$; since $\R\Gamma_\n$
commutes with the restriction functor, the cohomological dimension of $\Gamma_\n$ over
$\Z\GrMod B^e$ is bounded by this number.

If $\lcd_{\n,\Z} B = d < \infty$ then $\R\Gamma_\n(B) \in \D^b(\Z\GrMod B^e)$. Indeed,
let $I^\bullet$ be an injective resolution of $B$ in $\Z\GrMod B^e$. Truncating at
position $d$ we get the complex
\[
J^\bullet = 0 \longrightarrow I^0 \longrightarrow I^1 \longrightarrow \dots
\longrightarrow I^{d-1} \longrightarrow \ker(f^d) \longrightarrow 0 \longrightarrow \dots
\]
which is a $\Gamma_\n$-acyclic resolution of $B$ in $\Z\GrMod B^e$, so $\R \Gamma_\n(B)
\cong \Gamma_\n(J^\bullet)$ is a bounded complex.
\end{subremark}

The following property was introduced in the seminal paper [AZ]. 
\begin{subdefinition} 
The algebra $B$ is said to have property $\chi$ (on the left) if for every finitely
generated object $M$ of $\Z\GrMod B $, $\EXT_B^i(\k,M)$ has right bounded grading for all
$i\in\N$. Further, $B$ is said to satisfy property $\chi$ on the right if $B^\opp$
satisfies property $\chi$.
\end{subdefinition}

The following observations on property $\chi$ will be useful later. 
For proofs and more information, the reader is referred to section 3 in [AZ].

\begin{subremark} -- \label{ki-et-dimension} \rm
Recall that a $\Z$-graded $\k$-vector space is said to be locally finite if each of its
homogeneous components is a finite dimensional $\k$-vector space. Since $B$ is connected
and noetherian, it is locally finite, and so is any object of $\Z\GrMod B $. If $N$ and
$M$ are objects of $\Z\GrMod B $ with $N$ finitely generated and $M$ left bounded and
locally finite, then for each $i \geq 0$ the $\k$-vector spaces $\EXT_{\Z\GrMod B }^i(N,M)$ are
left bounded and locally finite. It follows at once that $B$ satisfies property $\chi$ if
and only if for every object $M$ of $\Z\GrMod B $ and each $i \geq 0$ the $\k$-vector space
$\EXT_B^i(\k,M)$ is finite dimensional.

As stated in [AZ, Corollary 3.6], $B$ satisfies property $\chi$ if and only if, for every
object $M$ of $\Z\GrMod B $, the cohomology module $H_\m^i(M)$ has right bounded grading
for all $i\in\N$. Notice that property $\chi^\circ$ mentioned in the reference is
equivalent to $\chi$ because $B$ is locally finite, see [AZ, Proposition 3.11 (2)].
\end{subremark}

The following is a result relating property $\chi$ with the existence
of a balanced dualizing complex for $B$, see [VdB; Theorem 6.3].

\begin{subtheorem} -- {\bf Existence criterion for balanced dualizing complexes}
\label{critere-vdb}\\
The algebra $B$ has a balanced dualizing complex if and only if both $B$ and $B^\opp$ have
finite local cohomological dimension and satisfy property $\chi$.
\end{subtheorem}

We will use Van den Bergh's criterion to show that $A$ has a balanced dualizing complex if
and only if any twist ${}^\tau A$ has a balanced dualizing complex. For this we need the
following result.
\begin{subproposition} -- \label{CNS-chi}
Suppose $B$ has finite local cohomological dimension. In that case $B$ has property $\chi$
if and only if the $\Z$-graded $\k$-vector spaces $\EXT_{B}^i(\k,B)$ are finite
dimensional for all $i \geq 0$.
\end{subproposition}
\begin{proof} 
We have already stated that property $\chi$ implies that $\EXT_B^i(\k,M)$ is finite
dimensional for all $i$, so the ``only if'' part is clear.

For the ``if'' part, we use the local duality theorem for connected graded algebras, see
[VdB; Theorem 5.1] or [Jor; Theorem 2.3]: under the hypothesis, for
every object $M$ of $\Z\grmod B$ there exists an isomorphism
\[
\R\Gamma_\n(M)' \cong \R\HOM_B(M, \R\Gamma_\n(B)')
\]
in $\D(\Z\GrMod B^\opp)$. Since $B$ has property $\chi$ if and only if $H^i_\n(M)$ has
right bounded grading for all $M$ in $\Z\grmod B$, it is enough to check that the
cohomology modules of the complex on the left hand side of this isomorphism have
\emph{left} bounded grading.

Let $\A$, resp $\B$, be the category $\Z\GrMod B^e$, resp. $\Z\GrMod B^\opp$, and let
$\A'$ and $\B'$ be the corresponding thick subcategories consisting of objects with
left bounded grading. For every object $M$ of $\Z\GrMod B $ we consider the functor
\[
\R\HOM_B(M,-) : \D^+(\A) \longrightarrow \D(\B).  
\]
Let $X$ be an object of $\A'$. We can compute $\EXT^i_{B}(M,X)$ using a
finitely generated free resolution of $M$, say $P^\bullet$. Since $B$ has left bounded
grading, the same holds for all the $B^\opp$-modules of the complex $\HOM_B(P^\bullet,
M)$, and hence also for $\EXT^i_B(M, X)$, being a subquotient of $\HOM_B(P^{-i}, X)$; in
other words $\EXT^i_B(M, X)$ lies in $\B'$. By [Hart; Proposition 7.3(i)], given any
complex $X^\bullet$ with cohomology modules in $\A'$, the cohomology modules of
$\R\HOM_B(M, X^\bullet)$ are in $\B'$.

By Remark \ref{flcd-et-borne}, $\R\Gamma_\n(B)$ and $\R\Gamma_\n(B)'$ lie in
$\D^b(\Z\GrMod B^e)$. The hypothesis on $B$ together with [AZ;
corollary 3.6(3)] ensures that the cohomology modules of this complex lie in $\A'$, so the
cohomology modules of $\R\HOM_B(M, \R\Gamma_\n(B)') \cong \R \Gamma_\n(M)'$ lie in $\B'$,
that is they have left bounded grading.
\end{proof}

We are now ready to prove that having a balanced dualizing complex is a twisting-invariant
property. Once again we regard $A$ as being $\N$-graded connected with the grading induced
by $\phi$.
\begin{subproposition} -- \label{tranfert-bdc}
Let $\tau$ be a normalized left twisting system on $A$ over $\Z^{r+1}$. 
\begin{enumerate}
\item $\lcd_{\m,\Z}(A) = \lcd_{{}^\tau \m, \Z}({}^\tau\!A)$.

\item Suppose $\lcd_{\m, \Z} A$ is finite. In that case $A$ has property $\chi$ if and only if
${}^\tau\! A$ has property $\chi$.

\item $A$ has a balanced dualizing complex if and only if $\,{}^\tau\! A$ has a balanced
dualizing complex.
\end{enumerate}
\end{subproposition}
\begin{proof}
By Corollary \ref{lcd-group-change} we know that $\lcd_{\m, \Z} A = \lcd_{\m,
\Z^{r+1}} A$ and the same holds for ${}^\tau A$, so item 1 follows from Lemma
\ref{twisting-torsion}. By Proposition \ref{extension-and-!-and-*} and Lemma
\ref{iso-pre-chi}, for every $i \geq 0$, the $\k$-vector spaces $\EXT_{\Z\GrMod A}^i(\k, A),
\EXT_{\Z^{r+1}\GrMod A}^i(\k, A), \EXT_{\Z\GrMod {}^\tau A}^i(\k, {}^\tau A)$ and
$\EXT_{\Z^{r+1}\GrMod {}^\tau A}^i(\k, {}^\tau A)$ are isomorphic, so item 2 follows
from Proposition \ref{CNS-chi}. Finally item 3 follows from 1 and 2 by Van den Bergh's
criterion, Theorem \ref{critere-vdb}.
\end{proof}

\section{Twisted semigroup algebras.}\label{tsa}
In this section, we introduce some specific types of noncommutative deformations of
semigroup algebras. The study of these algebras is the main goal of this work. Although
the hypothesis is not always necessary, we will always assume that semigroups are
commutative, cancellative and have a neutral element; accordingly, subsemigroups must
contain the neutral element and we only consider morphisms compatible with this structure.
Any semigroup $S$ with these properties has a group of fractions $G$, obtained by adding
formal inverses to all elements by a process completely analogous to that of passing from
$\Z$ to $\QQ$, with the property that any semigroup morphism $S \to H$ with $H$ a group
factors through $G$. We will often identify a semigroup with its image inside its group of
fractions. The definition of a $G$-graded $\k$-algebra for a given commutative group $G$
extends in an obvious way to the notion of an $S$-graded $\k$-algebra, and any $S$-graded
$\k$-algebra can be seen as a $G$-graded $\k$-algebra, $G$ being the group of fractions of
$S$, where the homogeneous component of an element of $G$ not in $S$ is trivial.
	
The section is organized as follows. In subsection \ref{ss-tsa} we define certain
noncommutative deformations of a semigroup algebra by means of $2$-cocycles over the
corresponding semigroup. In subsection \ref{ss-tasa} we focus on the case where the
semigroup is a finitely generated sub-semigroup of $\Z^{n+1}$ for some $n \in \N$, and
apply the results obtained in the previous sections to establish results regarding their
homological regularity properties; we also prove some ring theoretical properties of these
algebras. In section \ref{ss-tla} we study \emph{twisted lattice algebras}, which were
introduced in [LR1] as degenerations of quantum analogues of Schubert varieties, and show
that they fall under the scope of the previous subsection. \\

\subsection{Twisted semigroup algebras.}\label{ss-tsa}
Let $(S,+)$ be a commutative cancellative semigroup and let $G$ be its group of fractions.
Recall that the semigroup algebra $\k[S]$ is by definition the $\k$-vector space with
basis $\{X^s \mid s \in S\}$ and  multiplication defined on the generators by $X^s \cdot
X^t = X^{s+t}$ and extended bilinearly. The semigroup algebra is a $G$-graded
$\k$-algebra, where the homogeneous component of degree $s$ of $\k[S]$ is $\k X^s$ if
$s\in S$ and $\{0\}$ if $s\in G\setminus S$.

We are interested in the possible associative unitary integral algebra structures one can
give to the underlying $\k$-vector space of $\k[S]$ which respect its $G$-grading. Given $s,t
\in S$ and any such product $*$ over $\k[S]$,
\[
X^s * X^t = \alpha(s,t) X^{s+t} \qquad \qquad (\dagger)
\]
where $\alpha(s,t)\in\k^*$, so $*$ induces a map $\alpha: S \times S \longrightarrow
\k^*$. We will prove that this is enough to guarantee integrality under a
mild hypothesis on the semigroup. In order to guarantee associativity $\alpha$ must
fulfill a $2$-cocycle condition: for any three elements $s, t, u \in S$ we must have 
$\alpha(s,t)\alpha(s+t,u) = \alpha(t,u) \alpha(s,t+u)$. Conversely any $2$-cocycle over
$S$ with coefficients in $\k^*$ defines an associative product on the $\k$-vector space $\k[S]$
by means of formula $(\dagger)$. We denote by $\cdot_\alpha$ this product, and by
$\k^\alpha[S]$ the algebra thus obtained. This is a unitary algebra with $1_{\k^\alpha[S]}
= \alpha(0_S,0_S)^{-1} X^0$.

The $2$-cocycles of $S$ with coefficients in the group $\k^*$ form a group with pointwise
multiplication, denoted by $C^2(S,\k^*)$. The unit of this group is the map $\textbf{1} :
S \times S \longrightarrow \k^*$, defined by $(s,t) \mapsto 1$. Clearly
$\k^{\textbf{1}}[S] = \k[S]$.

\begin{subdefinition} -- 
Let $\alpha \in C^2(S,\k^*)$. We refer to $\k^{\alpha}[S]$ as the \emph{$\alpha$-twisting}
of $\k[S]$.
\end{subdefinition}

\begin{subremark} -- \rm  
A $2$-cocycle $\alpha$ of $S$ with coefficients in $\k^*$ will be called
\emph{normalized} if $\alpha(0,0) = 1$. For such a $2$-cocycle, the unit of $\k^\alpha[S]$
is $X^0$. We denote by $C^2_{\norm}(S,\k^*)$ the subgroup of $C^2(S, \k^*)$ consisting of
normalized $2$-cocycles.
Given $\alpha \in C^2(S, \k^*)$ and $a = \alpha(0,0)$, the $2$-cocycle $\alpha' =
a^{-1}\alpha$ is normalized, and we call it the \emph{normalization} of $\alpha$.
Multiplication by the constant $a \in \k^*$ provides an isomorphism of $S$-graded
$\k$-algebras between $\k^\alpha[S]$ and $\k^{\alpha'}[S]$, so without loss of generality
we may always consider normalized $2$-cocycles.
\end{subremark}

Given a function $f: S \to \k^*$, the associated coboundary $\partial f : S \times S \to
\k^*$ is defined by 
\[
\partial f(s,t) = \frac{f(s)f(t)}{f(s+t)}.
\]
This is always a $2$-cocycle. The $2$-coboundaries form a subgroup of $C^2(S,\k^*)$
denoted by $B^2(S,\k^*)$. A $2$-coboundary $\partial f$ is normalized if and only if
$f(0)=1$. The following lemma proves that the cohomology group $H^2(S, \k^*) = C^2(S,
\k^*) / B^2(S, \k^*)$ parametrizes the isomorphism classes of $G$-graded integral unitary
algebra structures over $\k[S]$.

\begin{sublemma} -- 
Given $\alpha, \beta \in C^2(S,\k^*)$, the algebras $\k^\alpha[S]$ and $\k^\beta[S]$ are
isomorphic as unitary $S$-graded $\k$-algebras if and only if there exists a
$2$-coboundary $\partial f$ such that $\alpha = (\partial f)\beta$.
\end{sublemma}

\begin{proof} First assume $\alpha,\beta$ are elements of $C^2(S,\k^*)$ such that $\alpha =
(\partial f) \beta$ for some $\partial f \in B^2(S,\k^*)$. The $S$-graded $\k$-linear
isomorphism $F : \k^\alpha[S] \longrightarrow \k^\beta[S]$ mapping $X^s \mapsto f(s)
X^s$ is an isomorphism of unitary $\k$-algebras.
Conversely, suppose $\alpha,\beta$ are elements of $C^2(S,\k^*)$ such that there exists an
isomorphism $F : \k^\alpha[S] \longrightarrow \k^\beta[S]$ of unitary, $S$-graded
$\k$-algebras. Then for each $s\in S$ the element $X^s$ must be an eigenvector of $F$
of non-zero eigenvalue, which we denote by $f(s)$. It is then easily seen that $\alpha =
(\partial f) \beta$. \end{proof}
 
Let $B^2_{\norm}(S,\k^*) = C^2_{\norm}(S,\k^*) \cap B^2(S,\k^*)$. Since constant
$2$-cocycles are $2$-coboundaries and every $2$-cocycle can be written as a constant
cocycle times a normalized one, we get $H^2(S, \k^*) \cong C_\norm^2(S, \k^*) / B_\norm(S,
\k^*)$. This is reflected in the following fact: if the $2$-cocycles in the previous lemma
are normalized, then the $2$-coboundary $\partial f$ is also a normalized $2$-cocycle.

We recall that the semigroup $S$ is said to be totally ordered if the underlying set of
$S$ has a total order with the following property: given $s, s' \in S$ such that $s
\leq s'$, then $s + t \leq s' + t$ for all $t \in S$. The following lemma is easily
proved using a ``leading term'' argument.

\begin{sublemma} -- \label{integral-domain} If $S$ is totally ordered then the algebra
$\k^\alpha[S]$ is an integral domain for all $\alpha \in C^2(S,\k^*)$.
\end{sublemma} 

For the purpose of the next lemma we need to recall the definition of a quantum affine
space. Let $l\in\N$ and consider a multiplicatively skew-symmetric matrix
$\q=(q_{ij})_{0\leq i,j \leq l}$ with entries in $\k^*$. The associated quantum affine 
space, denoted by $\k_\q[X_0,\dots,X_l]$, is the $\k$-algebra with generators
$X_0,\dots,X_l$ and relations $X_iX_j=q_{ij}X_jX_i$ for $0 \leq i,j \leq l$. Quantum
affine spaces are iterated Ore extensions of $\k$ and therefore noetherian.

\begin{sublemma} -- \label{noetherian}
If $S$ is finitely generated as a semigroup then $\k^\alpha[S]$ is a noetherian algebra
for all $\alpha \in C^2(S,\k^*)$.
\end{sublemma}

\begin{proof} Suppose $S$ is generated as a semigroup by $s_0, \ldots, s_l$ for some $l\in\N$.
The algebra $\k^\alpha[S]$ is generated as a $\k$-algebra by $X^{s_0},\dots,X^{s_l}$, and
for $0 \leq i,j \leq l$ we have $\alpha(s_i,s_j)^{-1} X^{s_i} \cdot_\alpha X^{s_j} =
X^{s_i+s_j} = \alpha(s_j,s_i)^{-1} X^{s_j} \cdot_\alpha X^{s_i}$, so $X^{s_i} \cdot_\alpha
X^{s_j} = \frac{\alpha(s_i, s_j)}{\alpha(s_j, s_i)} X^{s_j} \cdot_\alpha X^{s_i}$. Setting
$q_{ij} = \frac{\alpha(s_i, s_j)}{\alpha(s_j, s_i)}$ for $0 \leq i,j \leq l$ we obtain a
multiplicatively skew-symmetric matrix $\q=(q_{ij})_{i,j}$, and the assignation $X_i
\mapsto X^{s_i}$ defines a surjective morphism of $\k$-algebras $\k_\q[X_0,\dots,X_l]
\longrightarrow \k^\alpha[S]$; since $\k_\q[X_0, \ldots, X_l]$ is noetherian, so is
$\k^\alpha[S]$. \end{proof}

\begin{subremark} -- \rm \label{semigroup-versus-group}
By [CE; Chapter X, Proposition 4.1] the inclusion $i : S \to G$ induces an isomorphism
$i^*: H^2(G,\k^*) \to H^2(S,\k^*)$. Setting $\iota = i \times i : S \times S \to G
\times G$, this result implies that given $\alpha \in C^2(S,\k^*)$ one can always find
a 2-cocycle $\beta \in C^2(G,\k^*)$ such that $\alpha$ and $\beta \circ \iota$ are
cohomologous. Hence there is a chain of $\k$-algebra morphisms
  \[
    \k^\alpha[S] \stackrel{\cong}{\to} \k^{\beta \circ \iota}[S] \hookrightarrow
    \k^\beta[G],
  \]
with the last one given by the assignation $X^s \mapsto X^{i(s)}$. From this we deduce
that every twist of $\k[S]$ by a 2-cocycle is isomorphic to a $G$-graded subalgebra of a
twist of $\k[G]$ by a 2-cocycle. This construction works for any monoid that
embeds injectively in its fraction group. We will give an explicit proof of this for
a special class of semigroups in the next section, see Proposition
\ref{full-affine-embedding}.
\end{subremark}

\subsection{Twisted affine semigroup algebras.} \label{ss-tasa}
We now introduce our main object of interest, twisted affine semigroup algebras. These are
twisted semigroup algebras where $S$ is an \emph{affine} semigroup. We start by recalling
some facts on affine semigroups.

\begin{subdefinition} -- An \emph{affine semigroup} is a finitely generated semigroup $S$
	which is isomorphic as a semigroup to a sub-semigroup of $\Z^l$ for some $l \geq
	0$. An affine semigroup $S$ is \emph{positive} if it is isomorphic to a
	sub-semigroup of $\N^l$ for some $l \geq 0$.
\end{subdefinition}

Let $S$ be an affine semigroup. By definition there exist $l\in\N$ and an injective
morphism $i: S \longrightarrow \Z^l$ of semigroups, so $S$ is abelian and cancellative.
Its group of fractions $G$ identifies with a subgroup of $\Z^l$, so it is isomorphic as a
group to $\Z^r$ for some $r \geq 0$.  Clearly such an integer $r$ is independent of the
choice of $l$ and $i$; it is called the \emph{rank} of $S$. An embedding $S
\stackrel{i}{\longrightarrow} \Z^l$ is called a \emph{full embedding} whenever the group
generated by the image of $S$ is $\Z^l$. Clearly any affine semigroup of rank $r$ has a
full embedding in $\Z^r$.

\begin{subdefinition} -- \label{def-tasa}
A twisted affine semigroup algebra is an algebra $\k^\alpha[S]$ where $S$ is an 
affine semigroup and $\alpha$ an element of $C^2_\norm(S,\k^*)$.
\end{subdefinition}

\begin{sublemma} -- \label{integre-et-noetherien}
A twisted affine semigroup $k^\alpha[S]$ is a noetherian integral domain.
\end{sublemma}

\begin{proof} Suppose the rank of $S$ is $r$ and fix an embedding $S \to \Z^r$. We can
endow $S$ with a total order by pulling back the lexicographic order of $\Z^r$; by Lemma
\ref{integral-domain} $\k^\alpha[S]$ is an integral domain. Since $S$ is finitely
generated, Lemma \ref{noetherian} implies $\k^\alpha[S]$ is noetherian. \end{proof}

\begin{subremark} -- \label{rmk-grading-tasa} \rm
Let $S$ be an affine semigroup and let $\alpha \in C^2_\norm(S,\k^*)$. 
Definition \ref{def-tasa} associates to the pair $(S,\alpha)$ the twisted affine semigroup
algebra $\k^\alpha[S]$. As discussed in the introduction to this section, $\k^\alpha[S]$
has a natural $G$-grading where $G$ is the group of fractions of $S$. Fixing an
isomorphism $G \cong \Z^r$, with $r$ the rank of $S$, we obtain a full embedding $\iota: S
\to \Z^r$. From this point on we assume that any affine semigroup comes with one such
embedding, so given any $\alpha \in C^2_\norm(S,\k^*)$ we will be in
position to associate to the triple $(S,\iota,\alpha)$ the $\Z^r$-graded twisted affine
semigroup algebra $\k^\alpha[S]$. The corresponding grading will be called the grading of
$\k^\alpha[S]$ associated to $\iota$, or simply the natural grading.
\end{subremark}

We now describe a way to produce twisted affine semigroup algebras as subalgebras of
quantum tori. Fix $l\in\N$ and a skew-symmetric matrix $\q=(q_{ij})_{0\le i,j \le l}$ with
entries in $\k^*$. Consider the associated quantum affine space $\k_\q[X_0,\dots,X_l]$.
As already noted this is a noetherian integral domain and the generators $X_0,\dots,X_l$
are normal regular, so we may form the localization of $\k_\q[X_0,\dots,X_l]$ at the
multiplicative set generated by $X_0,\dots,X_l$ which we denote by $\k_\q[X_0^{\pm
1},\dots,X_l^{\pm 1}]$ and call the \emph{quantum torus} associated to $l$ and $\q$. In
this context, we fix the following notation: for $s=(s_0,\dots,s_l) \in\Z^{l+1}$ we write
$X^s$ for $X_0^{s_0} \cdots X_l^{s_l} \in \k_\q[X_0^{\pm 1},\dots,X_l^{\pm 1}]$. Clearly
the set $\{X^s \mid s\in\Z^{l+1}\}$ is a basis of the $\k$-vector space $\k_\q[X_0^{\pm
1},\dots,X_l^{\pm 1}]$. We consider two gradings on this algebra: the first one is a
$\Z^{l+1}$-grading obtained by assigning degree $e_i$ to $X_i$ for all $0 \le i \le l$,
where $\{e_0,\cdots,e_l\}$ is the canonical basis of $\Z^{l+1}$. The second is a
$\Z$-grading obtained by assigning degree $1$ to each $X_i$. 

Let $S$ be any finitely generated sub-semigroup of $\Z^{l+1}$. It is clear that the
$\k$-subspace of $\k_\q[X_0^{\pm 1},\dots,X_l^{\pm 1}]$ generated by $\{X^s \mid s\in S\}$
is a $\k$-subalgebra, which we denote by $\k_\q[S]$; this subalgebra inherits both the
$\Z^{r+1}$ and the $\Z$-grading of the ambient quantum torus. Notice that for all $s,t \in
S$ there exists $\alpha(s,t) \in \k^*$ such that $X^s X^t = \alpha(s,t) X^{s+t}$, so
associated to this algebra we obtain a map $\alpha: S \times S \longrightarrow \k^*$. It
follows from the associativity of the product of $\k_\q[S]$ that $\alpha$ is
a $2$-cocycle over $S$ with coefficients in $\k^*$ and so $\k_\q[S]$ is isomorphic to
$\k^\alpha[S]$, that is $\k_\q[S]$ is a twisted affine semigroup algebra. The next
proposition shows that all twisted affine semigroup algebras arise this way. 

\begin{subproposition} -- \label{full-affine-embedding} 
Let $S$ be an affine semigroup, $\alpha \in C^2_\norm(S,\k^*)$, and $\iota: S
\longrightarrow \Z^{n+1}$ a full embedding of $S$. There exists a multiplicatively
skew-symmetric matrix $\q=(q_{ij})_{0 \le i,j \le n}$ and an injective morphism of
$\k$-algebras
\[
\begin{array}{ccrcl}
 & & \k^\alpha[S] & \longrightarrow & k_\q[X_0^{\pm 1}, \ldots, X_n^{\pm 1}] \cr 
\end{array}
\]
which for every $s \in S$ sends $X^s$ to a nonzero scalar multiple of $X^{\iota(s)}$, and
whose image is $\k_\q[\iota(S)]$.
\end{subproposition}

\begin{proof} 
Since $\iota$ is a full embedding, the rank of $S$ is $n+1$. Put $A=\k^\alpha[S]$.
Clearly for every $s \in S$ the element $X^s$ is normal, so we may consider the
localization of $A$ at the multiplicative set generated by all these elements, which we
denote by $A^\circ$. As discussed before $A$ has a $\Z^{n+1}$ grading induced by $\iota$,
and since we have only inverted homogeneous elements, so does $A^\circ$. By Lemma
\ref{integre-et-noetherien} every element we have inverted is normal regular, so the
natural map $A \to A^\circ$ is an injective morphism of $\Z^{n+1}$-graded $\k$-algebras.
We will finish the proof by showing that $A^\circ$ is isomorphic to a quantum torus in
$n+1$-variables.

Given $u \in \Z^{n+1}$, the homogeneous component $A^\circ_u$ is generated by all
monomials of the form $X^s(X^t)^{-1}$ with $\iota(s) - \iota(t) = u$; notice that there
always exist $s,t \in S$ with this property since $\iota$ is a full embedding. A simple
algebraic manipulation shows that any two such monomials are scalar multiples, so $\dim_\k
A^\circ_u = 1$ for all $u \in \Z^{n+1}$.

For each element $e_i$ of the canonical basis of $\Z^{n+1}$ we choose elements $s_i, t_i
\in S$ such that $\iota(s_i)-\iota(t_i)=e_i$, and write $Y_i:= X^{s_i}(X^{t_i})^{-1}$ for
$0 \leq i \leq n$. Since the elements $X^s$ commute up to nonzero scalars with each other,
the same must be true of the $Y_i$, so there exist $q_{ij} \in\k^*$ such that $Y_i Y_j =
q_{ij} Y_j Y_i$; since $A^\circ$ is an integral domain, $\q = (q_{ij})$ is a
multiplicatively skew-symmetric matrix. In addition, ordered monomials in $Y_0,\dots,Y_n$
with powers in $\Z$ form a basis of homogeneous elements for $A^\circ$, so there is a
$\Z^{n+1}$-graded $\k$-algebra morphism $\psi: \k_\q[X_0^{\pm 1},\cdots,X_n^{\pm 1}] \to
A^\circ$ such that $\psi(X_i) = Y_i$ for each $0 \leq i \leq n$; since both algebras have
one-dimensional homogeneous components, $\psi$ is an isomorphism, and clearly
$\psi(\k_\q[\iota(S)]) = A$. \end{proof}

Thanks to Proposition \ref{full-affine-embedding}, in order to study twisted affine
semigroup algebras we can restrict to the following setting: we start with $n\in\N$, a
finitely generated subsemigroup $S$ of $\Z^{n+1}$
and a multiplicatively skew-symmetric matrix $\q=(q_{ij})_{0 \le i,j \le n}$ with entries
in $\k^*$. We then consider the $\k$-subalgebra $\k_\q[S]$ of $\k_\q[X_0^{\pm
1},\cdots,X_n^{\pm 1}]$, equipped with the $\Z^{n+1}$-grading it inherits from
$\k_\q[X_0^{\pm 1},\cdots,X_n^{\pm 1}]$.

\begin{subproposition} -- \label{tasa-est-un-zt}
In the setting of the last paragraph, the $\k$-algebra $\k_\q[S]$ is a left twist of
$\k[S]$ over $\Z^{n+1}$.
\end{subproposition}

\begin{proof}  
Using a reasoning similar to the one found in [Z; section 6] to study twists of quantum
affine spaces, one can prove that the $\Z^{n+1}$-graded algebra $\k_\q[X_0^{\pm 1},
\ldots, X_n^{\pm 1}]$ is a left twist of the Laurent polynomial algebra $\k[X_0^{\pm 1},
\ldots, X_n^{\pm 1}]$. The result follows by applying Proposition
\ref{Z-twist-on-subalgebras} to the subalgebra $\k[S]$ of $\k[X_0^{\pm 1}, \ldots,
X_n^{\pm 1}]$. \end{proof}

We are now in position to study the homological regularity properties of twisted affine
semigroup algebras for positive affine semigroups. This is done in Theorems
\ref{thm-AS-tasa} and \ref{thm-dc-tasa}. For both statements we fix an affine semigroup
$S$ with an embedding of semigroups $\iota : S \longrightarrow \Z^{n+1}$ such that
$\iota(S) \subseteq \N^{n+1}$. As specified in remark \ref{rmk-grading-tasa}, for
any $\alpha\in C^2(S,\k^*)$ the algebra $\k^\alpha[S]$ comes equipped with a natural
$\Z^{n+1}$-grading and is connected with respect to this grading, so we are in the context
of the introduction to section \ref{section-regularity-grading-twist}.

\begin{subtheorem} -- \label{thm-AS-tasa} 
With the previous notation and for all $\alpha \in C^2_\norm(S,\k^*)$, the following
holds:
\begin{enumerate}
	\item $\k^\alpha[S]$ is AS-Cohen-Macaulay if and only if $\k[S]$ is
		AS-Cohen-Macaulay.
	\item $\k^\alpha[S]$ is AS-Gorenstein if and only if $\k[S]$ is AS-Gorenstein.
	\item $\k^\alpha[S]$ is AS-regular if and only if $\k[S]$ is AS-regular.
\end{enumerate}
\end{subtheorem}

\begin{proof} Since $\k^\alpha[S]$ is a twist of $\k[S]$, the result follows at once from
Theorem \ref{stabilite-par-twist}. \end{proof}

\begin{subremark} -- \rm 
For a precise account on the regularity of affine semigroup algebras in the commutative
setting, the reader is referred to [BH; Chap. 6], in particular to statements 6.3.5 and
6.3.8.
\end{subremark}

\begin{subtheorem} -- \label{thm-dc-tasa}
In the above notation and for all $\alpha \in C^2_\norm(S,\k^*)$, the algebra
$\k^\alpha[S]$ has a balanced dualizing complex.
\end{subtheorem}

\begin{proof} Notice first that $\k[S]$ is a commutative noetherian connected $\N$-graded
$\k$-algebra of finite Krull dimension. By Grothendieck's vanishing theorem [BS; Theorem 
6.1.12] and Corollary \ref{lcd-group-change}, the local cohomological dimension of $\k[S]$ is
finite, and by [AZ; Prop. 3.11] it has property $\chi$. Thus Theorem \ref{critere-vdb} 
states that $\k[S]$ has a balanced dualizing complex. Since $\k^\alpha[S]$
is a twist of $\k[S]$, the result follows from Proposition \ref{tranfert-bdc}. \end{proof}

We finish this subsection by establishing certain ring theoretic properties of twisted
affine semigroup algebras. We show that, under a mild hypothesis on the underlying
semigroup, a twisted affine semigroup algebra can be written as the intersection of a
finite family of sub-algebras of its skew-field of fractions, each isomorphic to a twisted
semigroup algebra with underlying semigroup $\Z^n \oplus \N$, see Proposition
\ref{decomposition}. As a consequence, we get a characterization of those twisted affine
semigroup algebras which are maximal orders in their skew-field. For this we need to
recall some geometric information on affine semigroups, which we quote without proof. The
interested reader will find a more thorough treatment of the subject in [F; Chapter 1] or
[BH; Chapter 6].\\

Let $S$ be a finitely generated sub-semigroup of $\Z^{n+1}$ for some $n\geq 0$, and assume
that $S$ generates $\Z^{n+1}$ as a group. We may see $S$ as a subset of $\RR^{n+1}$
and consider $\RR_+ S \subset \RR^{n+1}$, the set formed by $\RR$-linear combinations of
elements of $S$ with coefficients in the set of non-negative real numbers, called
the \emph{cone} generated by $S$ in $\RR^{n+1}$. Notice that any generating set of the
semigroup $S$ generates $\RR^{n+1}$ as an $\RR$-$\k$-vector space and that, if
$\{s_1,\dots,s_l\}$ is such a generating set, then $\RR_+S = \{r_1s_1 + \ldots + r_ls_l
\mid (r_1,\dots,r_l)\in\RR_+^l\}$. In particular, $\RR_+S$ is a convex polyhedral cone in
the sense of [F; Section 1.2].

A \emph{supporting hyperplane} of $\RR_+S$ is a hyperplane that divides $\RR^{n+1}$ in two
connected components such that $S$, and hence $\RR_+ S$, is contained in the closure of
one of them. A \emph{face} of $\RR_+ S$ is the intersection of $\RR_+S$ with a supporting
hyperplane; a \emph{facet} is a face of codimension one, i.e. a face that generates a
hyperplane of $\RR^{n+1}$. We write $\tau < \RR_+ S$ if $\tau$ is a facet of $\RR_+S$, and 
$H_\tau$ for the unique supporting hyperplane which contains $\tau$.

Let $\tau < \RR_+ S$. We denote by $D_\tau$ the closure of the connected component of
$\RR^{n+1} \setminus H_\tau$ containing $S$. By [F; Section 1.2, Point (8)], if $\RR_+S
\neq \RR^{n+1}$ then
\[
\displaystyle \RR_+S = \bigcap_{\tau < \RR_+ S} D_\tau.
\]
For any facet $\tau$ of $\RR_+S$ let $S_\tau=D_\tau \cap \Z^{n+1}$. This is clearly a
sub-semigroup of $\Z^{n+1}$, and we always have
\[
S \subseteq \RR_+S \cap \Z^{n+1} = \left(\bigcap_{\tau < \RR_+ S} D_\tau\right) \cap
\Z^{n+1} = \bigcap_{\tau < \RR_+ S} S_\tau. \qquad \qquad (\ddagger)
\]
However, equality does not hold unless an additional assumption is made on $S$.
\begin{subdefinition} -- 
Let $G$ be the group of fractions of $S$. We say that $S$ is
\emph{normal} if it satisfies the following condition: if $g \in G$ and there exists $p
\in \N^*$ such that $p g \in S$, then $g \in S$.
\end{subdefinition}
With the notation of the previous discussion, the semigroup $S_\tau$ is always normal,
hence so is $\bigcap_{\tau < \RR_+ S} S_\tau$. Thus for $S$ to be equal to this
intersection, it must be normal. The following proposition shows that this condition is
not only necessary but sufficient. Gordan's Lemma [BH; Proposition 6.1.2] states that if
$S$ is normal then $S = \RR_+S \cap \Z^{n+1}$, which combined with $(\ddagger)$ gives the
following result.
\begin{subproposition} -- 
\label{intersection-decomposition}
Let $S$ be a finitely generated subsemigroup of $\Z^{n+1}$ that generates $\Z^{n+1}$ as a
group. If $S$ is normal, then
\[
S = \displaystyle \bigcap_{\tau < \RR_+ S} S_\tau.
\]
\end{subproposition}

The previous proposition shows that one may recover a normal affine semigroup from the
facets of its cone. The following is a related result that characterizes the semigroup
$S_\tau$.
\begin{subproposition} -- 
\label{structure-s-tau}
Let $S$ be a finitely generated subsemigroup of $\Z^{n+1}$ which generates $\Z^{n+1}$ as a
group. If $\tau$ is any facet of $\RR_+S$, then $S_\tau \cong \Z^n \oplus \N$ as a
semigroup.
\end{subproposition}

\begin{proof}
Let $\tau < \RR_+S$ and let $S_\tau^*$ be the set of invertible elements of $S_\tau$; 
notice that $S_\tau^* = H_\tau \cap \Z^{n+1}$, and that the inclusion $S_\tau^* \subset
S_\tau$ is strict, since otherwise $S$ would be contained in $H_\tau$. Clearly $S_\tau^*$
is a subgroup of $\Z^{n+1}$, and by [F; 1.2(2)] it generates $H_\tau$ as a $\k$-vector space,
so it is a free abelian group of rank $n$. Since $S_\tau$ is normal and generates
$\Z^{n+1}$, the factor group $\Z^{n+1}/S_\tau^*$ is torsion free and hence $S_\tau^*$ is a
direct summand of $S_\tau$; we fix a complement $C_\tau$. Since $S_\tau$ is normal so is
$C_\tau$, so this last one must be a normal affine semigroup of rank $1$. The result
follows from the fact that any normal affine semigroup of rank $1$ is isomorphic to either
$\N$ or $\Z$, and $C_\tau$ is not a group.
\end{proof}

The next observation will be useful latter.

\begin{subremark} -- \rm
\label{remarque-sur-generateurs}
Let $S$ be a finitely generated subsemigroup of $\Z^{n+1}$ that generates $\Z^{n+1}$ as a
group, with a set of generators $X = \{s_1,\dots,s_l\}$, and let $\tau$ be any facet of
$\RR_+S$. Recall from [F; 1.2 (2)] that $H_\tau$ is generated by $X \cap \tau$, so without 
loss of generality we may assume that $\{s_1, s_2, \ldots, s_n\}$ is a basis of $H_\tau$.
Since $S_\tau = D_\tau \cap \Z^{n+1}$, it 
follows that $\Z s_1 + \dots + \Z s_n + \N s_{n+1} + \dots + \N s_l \subseteq S_\tau$. We 
now prove that, if $S$ is normal, this inclusion is an equality.

Given $1 \leq k \leq l$, set $S_k = \Z s_1 + \dots + \Z s_k + \N s_{k+1} + \dots + \N
s_l$; a simple induction on $k$ shows that if $S$ is normal then so is $S_k$. Let $x \in
S_\tau$. Since $\{s_1, \ldots s_n\}$ is a linearly independent set, and $s_l \notin H_\tau$, 
the set $\{s_1, \ldots, s_n, s_l\}$ is a basis of $\QQ^{n+1}$. Since $x$ has 
integral coordinates, there exist $a_1,\dots,a_n, a_l\in\QQ$ such that $x = a_1 s_1 + 
\dots + a_n s_n + a_l s_l$, and since $x \in D_\tau$ it must be the case that $a_l \geq 0$. 
This implies there exists $c \in\N^*$ such that $cx \in S_n$, and since $S_n$ is normal, 
$x \in S_j$, i.e. $S_\tau \subseteq S_j$.
\end{subremark}

Propositions \ref{intersection-decomposition} and \ref{structure-s-tau} have the following
consequence.

\begin{subproposition} -- \label{decomposition}
Let $S$ be a finitely generated subsemigroup of $\Z^{n+1}$ which generates $\Z^{n+1}$ as a
group, and let $\q=(q_{ij})_{0 \le i,j \le n}$ be any skew-symmetric matrix with entries
in $\k^*$. 
\begin{enumerate}
	\item If  $S$ is normal then $\displaystyle \k_\q[S] = \bigcap_{\tau<\RR_+ S}
		\k_\q[S_\tau] \subseteq \k_\q[\Z^{n+1}]$.
	\item For all $\tau<\RR_+S$, there exists a skew-symmetric $(n+1) \times (n+1)$
		matrix $\q_\tau$ such that $\k_\q[S_\tau]$ is isomorphic to the subalgebra
		$\k_{\q_\tau}[\Z^n \oplus \N]$ of $\k_{\q_\tau}[\Z^{n+1}]$.
	\item If $S$ is normal, then for all $\tau<\RR_+S$ the algebra $\k_\q[S_\tau]$ is
		isomorphic to a left and a right localization of $\k_\q[S]$.
\end{enumerate}
\end{subproposition}
\begin{proof} Item 1 follows at once from Proposition \ref{intersection-decomposition}.
Let us prove item 2. By Proposition \ref{structure-s-tau}, $S_\tau$ is isomorphic to $\Z^n
\oplus \N$ as semigroup. Fix a semigroup isomorphism $\phi : \Z^n \oplus \N
\longrightarrow S_\tau$ and set $t_i=\phi(e_i)$, where $e_i$ is the $i$-th element of the
canonical basis of $\Z^{n+1}$. For $0 \le i,j \le n$, there is an element $q'_{ij}\in\k^*$
such that the equality $X^{t_i}X^{t_j}=q'_{ij}X^{t_j}X^{t_i}$ holds in $\k_\q[\Z^{n+1}]$,
and $\q_\tau=(q'_{ij})_{0 \le i,j \le n}$ is a skew-symmetric matrix with entries in
$\k^*$. The morphism of $\k$-algebras $\k_{\q_\tau}[\Z^n\oplus\N] \longrightarrow
\k_\q[S_\tau]$  $X_i \mapsto X^{t_i}$, which is evidently an isomorphism.

We now prove item 3. Remark \ref{remarque-sur-generateurs} implies that there is a set of
generators  $\{s_1,\dots,s_l\}$ of $S$ such that $S_\tau = (\Z s_1 + \dots + \Z s_j) + (\N
s_{j+1} + \dots + \N s_l)$ for some integer $j$. It follows immediately that
$\k_\q[S_\tau]$ is the localization of $\k_\q[S]$ at the multiplicative set generated by
$X^{s_1}, \ldots, X^{s_j}$. \end{proof}

By a well-known result the ring $\k[S]$ is normal, i.e. it is an integral domain which is
integrally closed in its field of fractions, if and only if the affine semigroup $S$ is
normal; see [BH; 6.1.4] for a proof. We now show that the result extends to the
noncommutative situation with a suitable generalization of normality. 

Let $R$ be a noetherian integral domain, and let $\Frac(R)$ be its skew-field of
fractions, in the sense of [MR; Chapter 1] or [McCR; Chapter 2, \S 1]. It is then
immediate that $R$ is an order of $\Frac(R)$ as in [MR; chap. I, \S 2]. By definition $R$
is a maximal order of $\Frac(R)$ if it satisfies the following condition: if $T$ is any
subring of $\Frac(R)$ such that $R \subseteq T \subseteq \Frac(R)$ and if there exist $a,b
\in R \setminus \{0\}$ such that $aTb \subseteq R$, then $T = R$. We refer the reader to
[MR; Chapter 1] for a general account on maximal orders. 

For every non-zero ideal $I \subseteq R$ we set ${\cal O}_{l}(I) = \{q \in \Frac(R) \mid qI
\subseteq I \}$ and ${\cal O}_{r}(I) = \{q \in \Frac(R) \mid Iq \subseteq I \}$, called
the left and right order of $I$ in $\Frac(R)$, respectively. By [MR; Chap. I, 3.1], $R$ is
a maximal order of $\Frac(R)$ if and only if ${\cal O}_{l}(I) = {\cal O}_{r}(I) = R$ for
any non-zero ideal $I$ of $R$. We use this criterion to prove the following auxiliary
result.
\begin{sublemma} -- \label{lemme-OM}
	Let $R$ be a noetherian integral domain and $\I$ a nonempty set of left and right
	Ore sets of $R$. Suppose that for every $O \in \I$ the left and right localization
	$R[O^{-1}]$ is a maximal order in $\Frac(R)$. If $R = \bigcap_{O \in \I}
	R[O^{-1}]$, then $R$ is a maximal order in $\Frac(R)$.
\end{sublemma}

\begin{proof} Put $Q=\Frac(R)$ and $R_O=R[O^{-1}]$ for each $O \in \I$; notice that these
are noetherian integral domains by [McCR; 2.1.16]. Let $I$ be a non-zero ideal of $R$.
Recall that both $IR_O$ and $R_OI$ are equal to the two-sided ideal of $R_O$ generated by
$I$, see [McCR; 2.1.16]. Now for each $q \in \O_l(I)$, we have $qI
\subseteq I$, so $qIR_O \subseteq IR_O$ and by the criterion mentioned in the preamble to
this lemma we get that
$q \in R_O$ for all $O\in\I$; the same argument holds for $q \in \O_r(I)$. Since $R =
\cap_{O\in\I} R_O$ this shows that $R=\O_r(I)=\O_l(I)$, and hence $R$ is a maximal order
in $Q$.
\end{proof} 

\begin{subcorollary} -- \label{corollaire-OM}
Let $S$ be an affine semigroup. The following statements are equivalent:
\begin{enumerate}[(i)]
	\item $S$ is normal.
	\item $\k^\alpha[S]$ is a maximal order in its division ring of
		fractions for each $\alpha \in C^2_\norm(S,k^*)$.
	\item There exists $\alpha \in C^2_\norm(S,\k^*)$ such that $\k^\alpha[S]$ is a
		maximal order in its division ring of fractions.
\end{enumerate}
\end{subcorollary}
\begin{proof} 
Fix an integer $n\in\N$ such that $S$ identifies with a sub-semigroup of $\Z^{n+1}$
which generates $\Z^{n+1}$ as a group. By Proposition \ref{full-affine-embedding}, given
any $\alpha\in C^2_\norm(S,\k^*)$, the algebra $\k^\alpha[S]$ identifies with $\k_\q[S]
\subseteq \k_\q[X_0^{\pm 1},\dots,X_n^{\pm 1}]$ for some skew-symmetric matrix $\q$. We will
make this identification without further comment.
 
We first show that (i) implies (ii). We are in position to apply Proposition
\ref{decomposition} so $\displaystyle \k_\q[S] = \bigcap_{\tau<\RR_+S} \k_\q[S_\tau]$; now
for each facet $\tau<\RR_+S$, the algebra $\k_\q[S_\tau]$ is isomorphic to a quantum
affine space localized at some of its canonical generators, so it is a maximal order by
[MR; Chapitre V, Corollaire 2.6 and Chapitre IV, Proposition 2.1]. Also $\k_\q[S_\tau]$ is
a localization of $\k_\q[S]$, so it is enough to apply Lemma \ref{lemme-OM}. Obviously
(ii) implies (iii).

We now prove that (iii) implies (i). Consider $t=(t_0,\dots,t_n) \in \Z^ {n+1}$ and
suppose that $kt \in S$ for some positive integer $k$. We denote by $T$ the left
$\k_\q[S]$-submodule of $\Frac(\k_\q[S])$ generated by the set $\{(X^t)^l \mid l\in\N\}$,
where $X^t = X_0^{t_0} \dots X_n^{t_n}$. Since $X^t$ commutes up to non-zero scalars with
all $X^s$ for $s\in S$, we see that $T$ is a subring of $\Frac(\k_\q[S])$ which clearly
contains $\k_\q[S]$. The hypothesis on $t$ implies that $T$ is finitely generated, since
$T = \sum_{l=0}^{k-1} \k_\q[S] (X^t)^l$. For each $0 \leq l < k$ fix $s_l, t_l \in S$ such
that $tl = s_l - t_l$, so $(X^t)^l$ and $X^{s_l}(X^{t_l})^{-1}$ coincide up to a non-zero
scalar. Taking $a=1$ and $b=\prod_{0 \le l < k} X^{t_l}$ we get that $a T b \subseteq
\k_\q[S]$; since $\k_q[S]$ is a maximal order in its division ring of fractions $T
\subseteq \k_\q[S]$, in particular $X^t \in \k_\q[S]$, so $t\in S$.
\end{proof}

\subsection{Twisted lattice algebras.}\label{ss-tla}

In this subsection we consider a class of algebras arising in a natural way from a given
finite distributive lattice together with some additional data. These algebras, which we
call \emph{twisted lattice algebras}, were introduced in [RZ; Section 2] where they
appeared as degenerations of some quantum Schubert and Richardson varieties; their study
was the original motivation for this work. In the course of this subsection we show that
they are twisted affine semigroup algebras, so the results from the previous subsection
apply to them.

For the basic notions concerning ordered sets and lattices, as well as all unexplained
terminology and notation, we refer to [RZ; Section 1] and the references therein. Recall
in particular the following classical result.

\begin{subtheorem} {\bf (Birkhoff)} -- \label{birkhoff}
Let $\Pi$ be a finite distributive lattice. Denote by $\Pi_0=\irr(\Pi)$ the set of its
join-irreducible elements and by $J(\Pi_0)$ the set of $\Pi_0$-ideals, ordered
by inclusion. The map
\[
\begin{array}{ccrcl}
\varphi & : & \Pi & \longrightarrow & J(\Pi_0) \cr
 & & \alpha & \longmapsto & \{\pi\in\Pi_0 \mid \pi\le\alpha\}   
\end{array}
\]
is an isomorphism of lattices, and the rank of $\Pi$ coincides with the cardinality of
$\Pi_0$.
\end{subtheorem}

For the rest of this sub-section we fix a finite distributive lattice $(\Pi, \leq)$. We
will associate to $\Pi$ a normal affine semigroup.
Let $\FrMon(\Pi)$ be the free commutative monoid over $\Pi$. For $x$ in $\FrMon(\Pi)$
different from the unit element we define the length of $x$, denoted $\ell(x)$, as the
unique element $l$ of $\N^*$ such that $x$ may be written as the product of $l$ elements
of $\Pi$. By convention the unit element has length $0$.

We consider the equivalence relation on $\FrMon(\Pi)$ compatible with the product
generated by the set $\{(\alpha\beta,(\alpha\meet\beta)(\beta\join\alpha)) \mid
\alpha,\beta\in\Pi\} \subseteq \FrMon(\Pi) \times \FrMon(\Pi)$. We denote by $\StMon(\Pi)$
the quotient monoid
\[
\StMon(\Pi) = \FrMon(\Pi)/\sim .
\]
which we call the \emph{straightening semigroup} of $\Pi$.

\begin{sublemma} -- \label{formal-monomial-basis}
Any element of $\StMon(\Pi)$ different from the unit element may be written as a product 
$\pi_1 \dots \pi_s$, with $s\in\N^*$ and $\pi_1 \le \dots \le \pi_s$.
\end{sublemma}

\begin{proof} The argument is an easy double induction on $s$ and the depth of $\pi_s$ in
$\Pi$, analogous to the proof of [RZ; Lemma 2.5]. \end{proof}

We now want to show that $\StMon(\Pi)$ is actually a normal affine semigroup. For this we
follow [H]. Let $n \in\N$ be the rank of $\Pi$. We consider on $\Pi_0 \subseteq \Pi$ the
induced order, and extend it a total order $\le_\tot$. We denote by $p_1 <_\tot \dots
<_\tot p_n$ the strictly increasing sequence of elements of $\Pi_0$ with respect to
$\le_\tot$, and consider the morphism of monoids
\[
\begin{array}{ccrcl}
 & & \FrMon(\Pi) & \longrightarrow & \Z^{n+1} \cr
	& & \pi &\longmapsto &\displaystyle e_0 + \sum_{\{i \mid p_i\le\pi\} } e_i,
\end{array}
\]
where $\{e_0,\dots,e_n\}$ is the canonical basis of $\Z^{n+1}$. It is clear that for
$\alpha,\beta\in\Pi$ the images of $\alpha\beta$ and
$(\alpha\meet\beta)(\alpha\join\beta)$ by this map coincide. As a consequence we get
an induced morphism \[i: \StMon(\Pi) \to \Z^{n+1}.\]

We consider now the following sub-semigroups of $\Z^{n+1}$. Set 
\[
	T=\{(s_0,\dots,s_n)\in\N^{n+1} \tq s_0 \geq \max\{s_1,\dots,s_n\}\}.
\]
Given $\alpha,\beta \in \Pi_0$ we say that $\beta$ is consecutive to $\alpha$, which we
denote by $\alpha\prec\beta$, if $\alpha < \beta$ and there is no $\gamma\in\Pi_0$ such
that $\alpha < \gamma < \beta$. Now for each pair $(p_i,p_j)$ of elements of $\Pi_0$ such
that $p_j$ is consecutive to $p_i$ we define
\[ 
S_{ij} = \{(s_0,\dots,s_n)\in\N^{n+1} \tq s_i \ge s_j\}. 
\]
Finally, set
\[
S = T \cap \left(\bigcap_{p_i \prec p_j} S_{ij}\right). 
\]
We will prove that the semigroup $\StMon(\Pi)$ is isomorphic to $S$.
\begin{subproposition} -- \label{strPI-is-nasg}
Keep the notation from the previous paragraph.
\begin{enumerate}[(i)]
	\item The map $i$ is injective and its image is equal to $S$.
	\item The semigroup $S$ generates $\Z^{n+1}$ as a group.
	\item The semigroup $\StMon(\Pi)$ is a normal affine semigroup of rank $n$. 
\end{enumerate}
\end{subproposition}
\begin{proof} By definition $S$ is a sub-semigroup of $\Z^{n+1}$, and since $i(\pi) \in S$
	for all $\pi\in\Pi$ we get that $i(\StMon(\Pi)) \subseteq S$.

To prove (i) we find an inverse to $i$. Let $s=(s_0,s_1,\dots,s_n)\in S$ and define the
support of $s$ by $\supp(s)=\{p_i, \, 1 \le i \le n \tq s_i \neq 0\}$. We easily see that
$\supp(s) \in J(\Pi_0)$, so whenever $s$ is a non-zero element of $S$ we may consider the
element $s'$ obtained from $s$ by subtracting $1$ from any nonzero entry of $s$, that is
$s'=s-i(\varphi^{-1}(\supp(s)))$. By definition $s'$ belongs to $S$, and its first entry
equals the first entry of $s$ minus $1$.

Now, let $s=(s_0,s_1,\dots,s_n)\in S$ be a non-zero element. Applying the construction of
the last paragraph inductively $s_0$ times we produce a sequence $s = s^{(0)}, s^{(1)},
\ldots, s^{(s_0)} = 0$ of elements of $S$ and a corresponding sequence
$\supp(s)=\supp(s^{(0)}) \ge \supp(s^{(1)})\ge\dots\ge \supp(s^{(s_0)})=\emptyset $ of
elements of $\Pi$. Thus
\[
s = i\circ\varphi^{-1}(\supp(s^{(0)})) + \dots + i\circ\varphi^{-1}(\supp(s^{(s_0-1)})).  
\]

We then define a map $\psi:S \longrightarrow\StMon(\Pi)$ as follows: we set $\psi(0) = 1$,
and for all $s\in S\setminus\{0\}$ set $\psi(s)=\varphi^{-1}(\supp(s^{(0)})) \dots
\varphi^{-1}\supp(s^{(s_0-1)})$; by the previous discussion $i \circ \psi = \id$. We now
prove that $\psi \circ i = \id$. In view of lemma \ref{formal-monomial-basis}, it is
enough to show that $\psi\circ i(\pi_1 \pi_2 \dots, \pi_s)=\pi_1 \pi_2\dots \pi_s$ for all
$s\in\N^*$ and $\pi_1, \dots, \pi_s\in\Pi$ such that $\pi_1 \le \dots \le \pi_s \in \Pi$.
An elementary proof shows that in this context, $i(\pi_1 \dots \pi_s)$ is a nonzero
element of $S$ to which the process described before associates the decreasing sequence
$\pi_s \ge \dots \ge \pi_1$ in $\Pi$, so $\psi\circ i(s)=s$.

We now show that $S$ generates $\Z^{n+1}$ as a group. It is clear that $e_0\in S$ and
that, given $1 \le i\le n$, the sets $I=\{\pi\in\Pi_0 \tq \pi\le p_i\}$ and
$J=\{\pi\in\Pi_0 \tq \pi < p_i\}$ are poset ideals of $\Pi_0$. Put $\pi_I=\varphi^{-1}(I)$
and $\pi_J=\varphi^{-1}(J)$; as we have seen $i(\pi_I),i(\pi_J)\in S$ and
$e_i=\varphi(\pi_J)-\varphi(\pi_I)$, which proves item (ii). Item (iii) follows
immediately from the previous items and the definition of $S$. \end{proof}

By an argument similar to the one used in the proof of item (i) of Proposition
\ref{strPI-is-nasg} we get the following statement, which strengthens the existence
statement of Lemma \ref{formal-monomial-basis}.

\begin{subproposition} -- \label{base-pour-alg-str} Keep the notation from the previous
	paragraph. To any element $a$ of $\StMon(\Pi)$ different from the unit element we
	may associate in a unique way an integer $t \in\N^*$ and an increasing sequence
	$\pi_1 \le \dots \le \pi_t \in\Pi$ such that $a=\pi_1\dots\pi_t$.
\end{subproposition}

\begin{subdefinition} -- 
Let $\Pi$ be a finite distributive lattice and $\alpha$ a normalized $2$-cocycle on
$\StMon(\Pi)$. We call the $\k$-algebra $\k^\alpha[\StMon(\Pi)]$  the twisted lattice
algebra associated with $\Pi$ and $\alpha$.
\end{subdefinition}

\begin{subexample} -- \label{qta-sont-qla} \rm
Let $\Pi$ be a finite distributive lattice and consider maps $\q : \Pi \times \Pi
\longrightarrow \k^*$ and $\c : \inc(\Pi \times \Pi) \longrightarrow \k^*$, where
$\inc(\Pi\times\Pi)$ is the set of pairs $(\alpha,\beta)\in\Pi\times\Pi$ such that
$\alpha$ and $\beta$ are incomparable elements of $\Pi$. To the data consisting of
$\Pi,\q$ and $\c$ we associate the {\em quantum toric algebra} $\A_{\Pi,\q,\c}$ defined
in [RZ; Section 2]. Assume that standard monomials on $\Pi$ form a $\k$-linear basis
of $\A_{\Pi,\q,\c}$ as in [RZ; Remark 2.1], and let $\psi: \A_{\Pi,\q,\c} \longrightarrow
\k[\StMon(\Pi)]$ be the $\k$-linear morphism that sends $1$ to $1$ and a standard
monomial $X_{\pi_1} \dots X_{\pi_t} \in \A_{\Pi,\q,\c}$ to the element $\pi_1 \dots \pi_t
\in \k[\StMon(\Pi)]$ for any $t\in\N^*$ and any increasing sequence $\pi_1 \le\dots\le
\pi_t \in\Pi$. The map $\psi$ is an isomorphism of $\k$-vector spaces by Proposition
\ref{base-pour-alg-str}, and so the product of $\A_{\Pi,\q,\c}$ induces a product on
$\k[\StMon(\Pi)]$, which is compatible with the $\Z^{n+1}$-grading induced by
the morphism $i: \StMon(\Pi) \to \Z^{n+1}$. Hence there exists a unique normalized
$2$-cocycle $\alpha$ on $\FrMon(\Pi)$ such that $\k[\StMon(\Pi)]$ endowed with this new
associative algebra structure equals $\k^\alpha[\StMon(\Pi)]$.
\end{subexample}

We are now in position to complete the proof of the result announced in [RZ; Remark 5.2.8]. 
For this, we use conventions and notation from [RZ].

\begin{subproposition} -- 
Suppose $A$ is a symmetric quantum graded algebra with a straightening law on the finite
partially ordered set $\Pi$, see [RZ; Definition 3.1], and suppose that $A$ satisfies
condition (C), see [RZ; Definition 4.1]. Then $A$ is a normal integral domain in the sense
of [RZ; Remark 5.2.8]. In particular, quantum Richardson varieties as defined in [RZ;
Definition 5.2.1] are normal domains.
\end{subproposition}

\begin{proof} By [RZ; Remark 5.2.8], it suffices to show that, if $\Pi$ is a finite
	distributive
lattice and $\q : \Pi \times \Pi \longrightarrow \k^*$ and $\c : \inc(\Pi \times \Pi)
\longrightarrow \k^*$ are maps such that standard monomials on $\Pi$ form a $\k$-linear
basis of $\A_{\Pi,\q,\c}$, then $\A_{\Pi,\q,\c}$ is a normal domain. As stated in Example
\ref{qta-sont-qla}, the algebra $\A_{\Pi,\q,\c}$ is isomorphic to a $2$-cocycle twist of
$\k[\StMon(\Pi)]$. On the other hand Proposition \ref{strPI-is-nasg} shows that
$\StMon(\Pi)$ is a normal affine semigroup, so applying Corollary \ref{corollaire-OM} we
prove the general statement. Since quantum Richardson varieties fall under this context as
stated in [RZ; Remark 5.2.8], the claim follows. \end{proof}

\section*{References.}
\noindent 
[AZ] 
M. Artin, J.J. Zhang. 
Noncommutative projective schemes. Adv. Math. 109 (1994), no. 2, 228--287.

\noindent
[Ber]
Berger, Roland(F-SETN-LAM)
Dimension de Hochschild des alg\`ebres gradu\'ees. (French. English, French summary) 
C. R. Math. Acad. Sci. Paris 341 (2005), no. 10, 597–600. 

\noindent
[BS]
M.P. Brodmann, R.Y. Sharp. 
Local cohomology: an algebraic introduction with geometric applications. 
Cambridge Studies in Advanced Mathematics, 60. Cambridge University Press, Cambridge.

\noindent
[BZ]
Dualising complexes and twisted Hochschild (co)homology for Noetherian Hopf algebras.
J. Algebra 320 (2008), no. 5, 1814–1850.

\noindent
[BH]
W. Bruns, J. Herzog. Cohen-Macaulay rings.
Cambridge Studies in Advanced Mathematics, 39,
Cambridge University Press,
Cambridge, 1993.

\noindent
[Cal] P. Caldero.
Toric degenerations of Schubert varieties. 
Transform. Groups 7 (2002), no. 1, 51--60. 

\noindent
[CE] H. Cartan, S. Eilenberg. 
Homological algebra. 
Princeton University Press, Princeton, N. J., 1956.

\noindent
[F] W. Fulton.
Introduction to toric varieties. Annals of Mathematics Studies,
131, The William H. Roever Lectures in Geometry,
Princeton University Press, Princeton, NJ, 1993.

\noindent
[Hart] Residues and duality.
R. Hartshorne. Residues and duality. Lecture Notes in Mathematics, No. 20 Springer-Verlag, Berlin-New York 1966.

\noindent
[H] T. Hibi. 
Distributive lattices, affine semigroup rings and algebras with straightening laws.  
Commutative algebra and combinatorics (Kyoto, 1985),  93--109, 
Adv. Stud. Pure Math., 11, North-Holland, Amsterdam, 1987.

\noindent
[Ing] C. Ingalls. : Quantum toric varieties, preprint available at 

\noindent
\emph{http://kappa.math.unb.ca/\~{}colin/research/nctoric.pdf}

\noindent
[Isch]
Eine Dualität zwischen den Funktoren Ext und Tor. (German)
J. Algebra 11 1969 510–531. 

\noindent
[Jor] P. J\o rgensen. 
Local cohomology for non-commutative graded algebras. 
Comm. Algebra 25 (1997), no. 2, 575--591. 

\noindent
[JZ] P. J\o rgensen, J. J. Zhang.
Gourmet's Guide to Gorensteinness
Advances in Mathematics, 01/2000; no. 151(2):313 -- 345. 

\noindent
[LR] Lakshmibai, V.(1-NORE); Reshetikhin, N.(RS-AOS2)
Quantum flag and Schubert schemes. (English summary) Deformation theory and quantum groups
with applications to mathematical physics (Amherst, MA, 1990), 145–181,
Contemp. Math., 134, Amer. Math. Soc., Providence, RI, 1992. 

\noindent
[LR1] T.H. Lenagan, L. Rigal. 
Quantum graded algebras with a straightening law and the AS-Cohen-Macaulay property for quantum 
determinantal rings and quantum Grassmannians. 
J. Algebra 301 (2006), no. 2, 670 -- 702.

\noindent
[Lev] T. Levasseur.
Some properties of noncommutative regular graded rings. 
Glasgow Math. J. 34 (1992), no. 3, 277--300.

\noindent
[MR] G. Maury, J. Raynaud. 
Ordres maximaux au sens de K. Asano. 
Lecture Notes in Mathematics, 808. 
Springer, Berlin, 1980.

\noindent
[McCR] 
J.C. McConnell, J.C. Robson. 
Noncommutative Noetherian rings. 
Graduate Studies in Mathematics, 30. 
American Mathematical Society, Providence, RI, 2001.

\noindent 
[NV] C. Nastasescu, F. van Oystaeyen.
Graded ring theory. 
North-Holland Mathematical Library, 28. North-Holland Publishing Co., Amsterdam-New York, 1982. 

\noindent
[RZ] L. Rigal, P. Zadunaisky.
Quantum analogues of Richardson varieties in the Grassmannian and their toric degeneration.
J. Algebra 372 (2012), 293--317. 

\noindent
[S] Soibelʹman, Ya. S.(RS-ROST)
On the quantum flag manifold. (Russian) Funktsional. Anal. i Prilozhen. 26 (1992), no. 3,
90--92; translation in Funct. Anal. Appl. 26 (1992), no. 3, 225–227.

\noindent
[VdB] M. van den Bergh. Existence theorems for dualizing complexes over non-commutative graded and filtered rings.
J. Algebra 195 (1997), no. 2, 662--679. 

\noindent [W] C.A. Weibel. 
An Introduction to Homological Algebra.
Cambridge Studies in Advanced Mathematics, 38. Cambridge University Press, Cambridge, 1994.

\noindent[Y] 
A. Yekutieli. 
Dualizing complexes over noncommutative graded algebras. 
J. Algebra 153 (1992), no. 1, 41--84.

\noindent
[Zad] P. Zadunaisky. Homological regularity properties of quantum flag varieties and
related algebras. PhD Thesis, university of Buenos-Aires and univ. Paris-Nord, 2014.
Available online at
\emph{http://cms.dm.uba.ar/academico/carreras/doctorado/TesisZadunaisky.pdf} 

\noindent
[Zak] Zaks, 
Injective dimension of semiprimary rings, J. Algebra 13 (1969), 73-89

\noindent [Z] J. Zhang.
Twisted graded algebras an equivalences of categories.
Proc. London Math. Soc. (3) 72 (1996), no. 2, 281--311. 

\vskip .5cm

\noindent
Laurent RIGAL, \\
Universit\'e Paris 13, LAGA, UMR CNRS 7539, 99 avenue J.-B. Cl\'ement, 93430 Villetaneuse, 
France; e-mail: rigal@math.univ-paris13.fr\\

\noindent
Pablo ZADUNAISKY, \\
Universidad de Buenos Aires, FCEN, Departamento de Matem\'aticas, Ciudad Universitaria - Pabell\'on I - 
(C1428EGA) - Buenos Aires, Argentina; e-mail: pzadub@dm.uba.ar
\end{document}